\documentclass[11pt,final]{article}
\usepackage{a4}
\usepackage{amsmath}%
\usepackage{amsthm}%
\usepackage{amstext}%
\usepackage{amssymb}%
\usepackage{showkeys}%
\usepackage{epsfig}%
\usepackage{cite}
\usepackage{tikz}
\usepackage{caption}
\usepackage{graphicx}
\usepackage{multirow}
\usepackage[ruled, vlined]{algorithm2e}
\usepackage{booktabs}
\usepackage{makecell}
\usepackage{floatrow}
\usepackage{comment}
\usepackage{listings}
\usepackage{subcaption}
\usepackage[color=mygray]{todonotes}
\usepackage[colorlinks]{hyperref}

\usepackage[headings]{fullpage}

\usepackage{fancyhdr}

\definecolor{mygray}{gray}{0.95}

\makeatletter
\newcommand{\customlabel}[2]{%
	\protected@write \@auxout {}{\string \newlabel {#1}{{#2}{\thepage}{#2}{#1}{}} }%
	\hypertarget{#1}{#2}
}
\makeatother
\newcommand{\labterm}[2]{\customlabel{#1}{\text{$#2$}}}

\newtheorem{theorem}{Theorem}

\newtheorem{corollary}[theorem]{Corollary}

\newtheorem{lemma}[theorem]{Lemma}

\newtheorem{remark}{Remark}

\numberwithin{equation}{section}
\numberwithin{theorem}{section}
\numberwithin{table}{section}
\numberwithin{figure}{section}
\numberwithin{remark}{section}

\usepackage{chngcntr}
\usepackage{apptools}
\AtAppendix{\counterwithin{theorem}{subsection}}
\AtAppendix{\counterwithin{equation}{subsection}}

\newcommand{\R}{\mathbb{R}}%
\newcommand{\N}{\mathbb{N}}%

\newcommand{\cG}{{\mathcal G}}
\newcommand{\cL}{{\mathcal L}}

\DeclareMathOperator*\gra{gra}

\DeclareMathOperator*\zer{zer}

\DeclareMathOperator{\Span}{span}%

\newcommand{\bra}{ \left \langle  }
\newcommand{\ket}{ \right \rangle  }

\title{Extra-Gradient Method with Flexible Anchoring: Strong Convergence and Fast Residual Decay}

\author{Radu I. Bo\c t\thanks{Faculty of Mathematics, University of Vienna, Oskar-Morgenstern-Platz 1, 1090 Vienna, Austria.\\ The research of RIB and EC has been supported by the Austrian Science Fund (FWF), projects W 1260 and P 34922-N, respectively. \texttt{email:} radu.bot@univie.ac.at, enis.chenchene@univie.ac.at} \and Enis Chenchene\footnotemark[1]}

\begin{document}
\maketitle

\begin{abstract}
    In this paper, we introduce a novel Extra-Gradient method with anchor term governed by general parameters. Our method is derived from an explicit discretization of a Tikhonov-regularized monotone flow in Hilbert space, which provides a theoretical foundation for analyzing its convergence properties. We establish strong convergence to specific points within the solution set, as well as convergence rates expressed in terms of the regularization parameters. Notably, our approach recovers the fast residual decay rate $O(k^{-1})$ for standard parameter choices. Numerical experiments highlight the competitiveness of the method and demonstrate how its flexible design enhances practical performance.
\end{abstract}



%\tableofcontents

\section{Introduction}

Tikhonov regularization has a long history in optimization. Originally introduced in the context of inverse problems as a powerful tool to stabilize ill-conditioned linear systems, it has quickly become popular as a simple mechanism to induce desirable properties in dynamical systems, such as strong convergence to a particular solution \cite{cps08} and acceleration \cite{bk24}. Let $H$ be a real Hilbert space and consider the following problem:
\begin{equation}\label{eq:intro_monotone_inclusion}
    \text{Find} \ x \in H: \ M(x) = 0\,,
\end{equation}
where $M\colon H\to H$ is a monotone and $L$-Lipschitz continuous operator. One might be tempted to approach \eqref{eq:intro_monotone_inclusion} by considering the asymptotic behaviour of the following dynamical system
\begin{equation}\label{eq:monotone_flow}
    \dot x(t) + M(x(t)) =0 \quad \text{for} \ t \geq t_0\,, 
\end{equation}
and $t_0>0$. However, trajectories of \eqref{eq:monotone_flow} are well known to only converge weakly and \emph{ergodically} to a zero of $M$ \cite{bb76}. Tikhonov regularization \cite{cps08} is one possible modification that allows us to overcome this issue. It consists in introducing a \emph{regularization} function $\varepsilon\colon [t_0, +\infty)\to (0, +\infty)$ and considering the regularized system
\begin{equation}\label{eq:tychonov_flow}
    \dot x(t) + M(x(t)) + \varepsilon(t) \left(x(t) - v\right) =0 \quad \text{for} \ t \geq t_0\,,
\end{equation}
where $v \in H$ is the so-called \emph{anchor term}. If $\varepsilon$ is constant, \eqref{eq:tychonov_flow} is a special case of \eqref{eq:monotone_flow} with respect to the strongly monotone operator $M_\varepsilon:= M + \varepsilon (I - v)$, $I \colon H\to H$ being the identity operator. This dynamical system does indeed allow us to get a converging trajectory, but the limit is a zero of $M_\varepsilon$, hence suggesting $\varepsilon(t)\to 0$ as $t \to +\infty$. If $\varepsilon$ vanishes too fast, e.g., $\varepsilon \in L^1([t_0, +\infty), \R)$, the system \eqref{eq:tychonov_flow} is qualitatively equivalent to \eqref{eq:monotone_flow}, see \cite{ap10}. It generates trajectories weakly converging to a zero of $M$ provided that, e.g., $M$ is demipositive \cite{bruck75}. The interesting case is thus
\begin{equation}\label{eq:intro_parameters}
    \lim_{t \to +\infty}  \varepsilon(t)=0\, \quad \text{and} \quad \int_{t_0}^{+\infty} \varepsilon(t) dt = +\infty\,.
\end{equation}
In this setting, strong convergence of $x(t)$ to the projection of $v$ onto $\zer M$ as $t \rightarrow +\infty$ has been established under a list of additional parameter assumptions. Based on previous results by Browder \cite{browder67} in 1967, Reich in the seventies obtained the first result in this direction in \cite{reich77}, which requires $\varepsilon$ to be nonincreasing. Israel and Reich in \cite{ir81} later replaced this assumption with $\frac{|\dot \varepsilon(t)|}{\varepsilon^2(t)} \to 0$ as $t \to +\infty$. However, the latter does not include choices such as $\varepsilon(t) = \frac{\alpha}{t}$ for $\alpha>0$ and all $t\geq t_0$. This case was covered by Cominetti, Peypouquet, and Sorin in \cite{cps08} replacing the condition with $\int_{t_0}^{+\infty} |\dot \varepsilon(t)|dt = +\infty$, and by Bo\c t and Nguyen in \cite{bk24} who instead proposed $\frac{|\dot \varepsilon(t)|}{\varepsilon(t)} \to 0$ as $t\to +\infty$.

In parallel and independently, Halpern in 1967 in \cite{halpern67} initiated a flourished research line on the so-called \emph{anchored} iterative methods to find fixed points of non-expansive operators. For a nonexpansive operator $T\colon H\to H$, these are numerical algorithms of the form
\begin{equation}\label{eq:intro_halpern}
    x^{k+1} := \varepsilon^k v + (1-\varepsilon^k)T(x^k)\, \quad \text{for all} \ k \geq 0,
\end{equation}
for $(\varepsilon^k)_k$ satisfying, analogously to \eqref{eq:intro_parameters}, at least
\begin{equation}\label{eq:intro_parameters_discrete}
 \lim_{k\to +\infty}\varepsilon^k = 0\, \quad \text{and} \quad \sum_{k = 0}^{+\infty} \varepsilon^k = +\infty\,,   
\end{equation}
among additional assumptions. Lions in \cite{lions77} in the seventies complemented \eqref{eq:intro_parameters_discrete} by imposing $\frac{|\varepsilon^{k+1}-\varepsilon^k|}{(\varepsilon^k)^2}\to 0$ as $k \to +\infty$. However, this condition does not allow for choices of the form $\varepsilon^k = \frac{\alpha}{k + \beta}$ for all $k \geq 1$ and $\alpha, \beta >0$. Wittman \cite{wittmann92} later closed this gap with the assumption: $\sum_{k=0}^{+\infty} |\varepsilon^{k+1}-\varepsilon^k| < +\infty$. Perhaps inspired by the continuous-time approach, Reich in \cite{reich94} established that strong convergence of $(x^k)_k$ also holds when $(\varepsilon^k)_k$ is nonincreasing. Leaving the monotonicity of the sequence, Xu in \cite{xu02} recognized that, in fact, a sufficient condition for the convergence of the sequence is $\frac{|\varepsilon^{k+1}-\varepsilon^k|}{\varepsilon^k}\to 0$ as $k \to +\infty$. This flagrantly evident connection between the two theories (Tikhonov regularization for continuous time dynamics and Halpern/anchoring mechanisms for discrete time schemes) has seemingly been elusive for almost $60$ years until the first attempts of reconciliation were established in \cite{bk24, jpr24}. In fact, \eqref{eq:intro_halpern} can be derived and analyzed as a discretization of the Tikhonov flow \eqref{eq:tychonov_flow}, see \cite{bk24}.

The growing interest in Tikhonov or anchoring mechanisms of type \eqref{eq:intro_halpern} is not only due to their desirable property of generating sequences that converge strongly to specific points in the set of solutions, but also to their inherent acceleration. As first observed by Sabach and Shtern in \cite{ss17}, the condition introduced by Xu in \cite{xu02} can also be used to obtain a rate of $O(k^{-1})$ for the fixed-point residual $\|x^k-T(x^k)\|$ as $k \rightarrow + \infty$ for the choices of $(\varepsilon^k)_k$ of the form $\varepsilon^k = \frac{2}{k + 2}$ for all $k \geq 0$. This result was further tightened by Lieder in \cite{lieder21}, who improved the constant by Sabach and Shtern by a factor of $4$ for $\varepsilon^k = \frac{1}{k+2}$ for all $k \geq 0$, and established a tightness result. Later, Ryu and Park in \cite{pr22} proved that this choice is indeed optimal with respect to a large class of methods in a sense that is very close in spirit to Nemirovski's classical result \cite{NemirovskiYudin1983} for convex optimization. Not surprisingly, these acceleration results can be adapted to the continuous time setting \cite{bk24, jpr24}.

\begin{table}[b]
\centering
\begin{tabular}{@{}lll@{}}
    \toprule
    \textbf{Method} & $ \left\{\begin{aligned}
		&y^{k+1} := x^k - \varepsilon^k (x^k - v) - \eta^k M(z^{k+1})\,,\\
		&x^{k+1} := x^k - \varepsilon^k (x^k - v) - \theta^k M(y^{k+1})\,.
	\end{aligned} \right.
    $ 
        & \makecell[l]{$\eta^k, \theta^k, \varepsilon^k \in \mathbb{R}$,\\ $z^{k+1}\in H$.}\\
    \midrule
    \textbf{Reference} & \textbf{Specification} & \textbf{Characteristics} \\[0.3cm]
    \cite{yr21}  \text{EAG}-C & $z^{k+1} = x^k\,, \ \eta^k = \theta^k = \eta\,,  \ \varepsilon^k = \frac{1}{k+2}$, &
     $\eta \in \big(0, \frac{1}{8L}\big]$  \\
     \cite{yr21} \text{EAG}-V & $z^{k+1} = x^k\,, \ \eta^k = \theta^k$ as in \eqref{eq:step-size_eag_v}, $\varepsilon^k = \frac{1}{k+2}$, & Improved rate over EAG-C \\
     \cite{lk21} \text{FEG} & $z^{k+1} = x^k\,, \ \eta^k = \frac{1}{L}\left(1- \frac{1}{k+1}\right)\,, \ \theta^k = \frac{1}{L}\,, \ \varepsilon^k = \frac{1}{k+1},$ & 
    Improved rate over EAG-V\\
    \cite{tl21} APV & $z^{k+1} = y^k\,, \ \eta^k = \theta^k$ as in \eqref{eq:step-size_apv}, $\varepsilon^k = \frac{1}{k+2},$ & \makecell[l]{One operator evaluation,\\ additional variable}\\
    \bottomrule
\end{tabular}
\caption{Summary of Extra-Anchored-Gradient and Popov's-type methods}
\label{tab:eag_type_methods}
\end{table}

The Tikhonov acceleration mechanism not only applies to enhance the worst-case performance of fixed-point iterations of nonexpansive operators, but also for Korpelevich's Extra-Gradient \cite{korpelevich1976} (EG) and Popov's \cite{popov1980} type methods for equations governed by monotone and $L$-Lipschitz continuous operators. Although the literature on Halpern methods is quite rich, this setting --- which results from \emph{explicit} discretizations of \eqref{eq:tychonov_flow} --- has been significantly less explored, with only a few contributions that we list in Table \ref{tab:eag_type_methods}. In particular, the family of EG-type methods encompasses the Extra-Anchored-Gradient (EAG) method with constant and non-constant step sizes (EAG-C) and (EAG-V), respectively, introduced by Yoon and Ryu in \cite{yr21}, and the Fast Extra-Gradient (FEG) method introduced by Lee and Kim in \cite{lk21}. The main drawback of EG-type schemes is the double evaluation of the operator, on both $x^k$ and on $y^{k+1}$. If storing one additional variable is possible, an alternative method is the Anchored Popov's (APV) method introduced in the working paper \cite{tl21}. In contrast to EAG and FEG, for APV one has $z^{k+1} = y^k$, hence requiring only one evaluation of $M$, but with one additional variable to store per iteration. All these methods exhibit the rate $\|M(x^k)\|=O(k^{-1})$ as $k \to +\infty$, thereby improving the $o(k^{-\frac{1}{2}})$ convergence rate for EG, which was established with computer-assisted proofs only for $\theta \in \big(0, \frac{1}{\sqrt{2}L}\big]$ by Gorbunov, Loizou and Gidel in \cite{glg22}. We anticipate that as a byproduct of our analysis, we can extend the proof to $\theta \in \big(0, \frac{1}{L}\big)$, cf.~Remark \ref{rmk:eg}. The fact that these methods share the same convergence rate of $O(k^{-1})$ as $k \to +\infty$ is not surprising, as they are all \emph{almost} equivalent to the Halpern method (applied to the resolvent of $M$), see \cite{yr22}.

In this paper, we further leverage the connection between Tikhonov regularization and anchored methods to design a new Extra-Anchored-Gradient type method for a general sequence of parameters $(\varepsilon^k)_k$. The proposed method reads as in Algorithm \ref{alg:gfeg}. In line with \cite{ss17, bk24}, we show that if \eqref{eq:intro_parameters_discrete}, 
\begin{equation}\label{eq:intro_epsdot_over_eps}
	\lim_{k\to +\infty} \frac{|\varepsilon^{k+1} - \varepsilon^k|}{\varepsilon^k} = 0\, \quad \text{as} \quad k \to +\infty\,,
\end{equation}
and $\theta \in \big(0, \frac{1}{L}\big)$, this method generates a sequence $(x^k)_k$ that converges strongly to the projection of $v$ onto the solution set of \eqref{eq:intro_monotone_inclusion} with a rate depending on the sequence $(\varepsilon^k)_k$. We show that with particular choices of $(\varepsilon^k)_k$ such as $\varepsilon^k = \frac{\alpha}{\theta(k + \beta)}$ with $\alpha > 1$ and $\beta> 0$, the rate particularizes to $O(k^{-1})$ as $k \to +\infty$, and discuss several other possible choices. We obtain the new method as an explicit discretization of \eqref{eq:tychonov_flow} and follow the continuous time approach to establish strong convergence of the iterates and derive convergence rates. A notable feature of our approach is that we obtain strong convergence of $(x^{k})_k$ without relying on the Halpern iteration as proposed for EAG and FEG in \cite{yr22}. Eventually, we demonstrate the practical efficiency of Algorithm \ref{alg:gfeg} with several numerical examples. These illustrate in particular that Algorithm \ref{alg:gfeg} can outperform other variants such as EAG, FEG and APV.

\begin{algorithm}[t]
	\caption{Generalized Extra-Anchored-Gradient method (G-EAG)} 
	\label{alg:gfeg}
	Pick $x^0 \in H$, $v \in H$, $\theta \in \big(0, \frac{1}{L}\big)$, $(\varepsilon^k)_k$ satisfying \eqref{eq:intro_parameters_discrete} and \eqref{eq:intro_epsdot_over_eps}, and\\
	\For{$k =0,1,\dots$}{
		\begin{equation*}
			\left\{\begin{aligned}
				&y^{k+1} := x^k -  \theta \varepsilon^k  (x^k - v) - \theta M(x^k) \\
				&x^{k+1} := x^k - \theta \varepsilon^{k+1}(x^{k+1}-v) - \theta M(y^{k+1}) 
			\end{aligned}\right.
	\end{equation*}}\vspace{0.2cm}
    %\hrule \vspace{0.2cm} 
    {\footnotesize \textbf{Note:} The method is indeed explicit --- bring $x^{k+1}$ on the left hand-side and divide by $1 + \theta \varepsilon^{k+1}$. \vspace{0.1cm}}
\end{algorithm}

\subsection{Contribution}\label{subsec11}
For reader's convenience, we summarize our contribution below:
\begin{itemize}
    \item We introduce Algorithm \ref{alg:gfeg} and show that if $(\varepsilon^k)_k$ satisfies \eqref{eq:intro_parameters_discrete} and \eqref{eq:intro_epsdot_over_eps}, then $(x^k)_k$ converges strongly to the projection of $v$ onto $\zer M$ with a rate depending on $(\varepsilon^k)_k$, cf.~Theorem \ref{thm:convergence_feg}.
    \item As a by-product of our analysis, we show that EG converges with a convergence rate of $o(k^{-\frac{1}{2}})$ as $k \rightarrow +\infty$ for all $\theta \in \big(0, \frac{1}{L}\big)$, complementing a result in \cite{glg22}, cf.~Remark \ref{rmk:eg}.
    \item We show that if $\varepsilon^k:=\frac{\alpha}{\theta(k + \beta)}$ with $\alpha > 1$ and $\beta > 0$ for all $k \geq 0$, the convergence rate particularizes to $O(k^{-1})$ as $k \rightarrow +\infty$, thus improving the $o(k^{-\frac{1}{2}})$ rate of EG and Popov's methods, cf.~Corollary \ref{cor:rate_linear}.
    \item We present numerical experiments that demonstrate the competitiveness or superiority of Algorithm \ref{alg:gfeg} over anchored EG and Popov's type methods. We further promote some feasible choices that enhance the convergence behaviour of the method.
\end{itemize}

\subsection{Background and Preliminaries}\label{subsec12}
We say that the operator $M\colon H\to H$ is $L$-Lipschitz continuous, where  $L \geq 0$, if
\begin{equation}
    \|M(x) - M(x')\| \leq L \|x-x'\| \, \quad \text{for all} \ x, x'\in H.
\end{equation}
We further suppose that $M$ is \emph{monotone}. A multi-valued operator $M\colon H\to 2^{H}$ is said to be monotone if
\begin{equation}
    \bra m - m', x - x' \ket \geq 0\, \quad \text{for all} \ (x, m)\,, (x', m')\in \gra M\,,
\end{equation}
where the set $\gra M\subseteq H\times H$ is the \emph{graph} of $M$ and is defined by $\gra M:=\{(x, m) \in H \times H \mid m \in M(x)\}$. A monotone operator is called \emph{maximal monotone} if its graph is maximal with respect to the inclusion of sets among the class of monotone operators. A single-valued monotone and Lipschitz continuous operator is also maximal monotone \cite[Corollary 20.28]{BCombettes}. In particular, $M$ is strong-to-weak continuous \cite[Corollary 21.21]{BCombettes}. We refer to \cite{BCombettes} for a comprehensive introduction to the theory of monotone operators in Hilbert spaces.

For a monotone and Lipschitz continuous operator $M\colon H\to H$, we are interested in finding a \emph{zero} of $M$, i.e.~an element in the set $\zer M :=\{ x \in H \mid M(x)=0\}$, which is sometimes referred to as a \emph{root} of $M$. This problem is ubiquitous in applications. If $M$ is allowed to be multi-valued, this encompasses from classical first-order optimality conditions for convex minimization problems (Fermat's principle) to generalized conditions of optimality for sums of non-smooth convex functions \cite{Chambolle2011, bredies2022graph}. Here, we are mainly interested in minimax games
\begin{equation}\label{eq:minmaxgame}
    \min_{x \in H} \max_{y \in K} \ \mathcal{L}(x, y)\,,
\end{equation}
where $\cL \colon H\times K \to \R$ is called \emph{Lagrangian}, and $H, K$ are Hilbert spaces, a problem that is gaining increasing interest in machine learning due to its applications in Adversarial Training \cite{goodfellow14}. We say that $\cL$ is \emph{convex-concave} if $x\mapsto\cL(x, y)$ is convex for all $y\in K$ and $y\mapsto \cL(x, y)$ is concave for all $x \in H$. A \emph{saddle point} of $\cL$ is any couple $(x^*, y^*)\in H\times K$ such that $\cL(x^*, y)\leq \cL(x^*, y^*) \leq \cL(x, y^*)$ for all $(x, y) \in H\times K$. When $\cL$ is differentiable and convex-concave, saddle points of $\cL$ correspond to  zeros of the so-called \emph{saddle operator}
\begin{equation}
    \mathcal{G}_{\cL}\begin{pmatrix}
        x \\ y 
    \end{pmatrix} := \begin{pmatrix}
        \nabla_x \cL (x, y)\\
        -\nabla_y \cL(x, y)
    \end{pmatrix}\, \quad \text{for} \ x \in H\,, y\in K\,.
\end{equation}
The saddle operator $\cG_{\cL}$ is, in fact, monotone \cite{rockafellar1970}. It is $L$-Lipschitz continuous if $\cL$ is $L$\emph{-smooth}, i.e., $\nabla \cL$ is $L$-Lipschitz continuous. Therefore, minimax games with $L$-smooth Lagrangians fit our problem formulation \eqref{eq:intro_monotone_inclusion} with $M = \cG_{\cL}$.

\subsection{Continuous Time Analysis}

In this section, we briefly recall the continuous time analysis of the Tikhonov flow \eqref{eq:tychonov_flow} with the careful adaptation recently introduced by Bo\c t and Nguyen in \cite{bk24}, which allows establishing a fast convergence rate for the residual. Despite our main focus being on numerical algorithms, it is important to include this analysis in this paper, since our discrete-time approach is profoundly based on the continuous-time approach, which emerges as an excellent assistant to:
\begin{enumerate}
	\item Guide the convergence proof: Our Lyapunov analysis follows exactly the arguments in the continuous time approach, which is significantly simpler than its discrete time counterpart. We shall say that our approach is \emph{continuous time assisted}.
	\item Design converging algorithms: It will become clear in our analysis that certain design choices of our methods are dictated by the Lyapunov analysis itself.  Specifically, these are often established to control undesired terms that do not appear in the continuous time approach.
\end{enumerate}
Let $M\colon H\to H$ be a monotone and $L$-Lipschitz continuous operator with $\zer M \neq \emptyset$, and consider the system \eqref{eq:tychonov_flow}.  We assume that the regularization function $\varepsilon \colon [t_0, +\infty) \to (0,+\infty)$ is continuously differentiable. For every starting point, the system admits a unique locally absolutely continuous solution $t \mapsto x(t)$, therefore $t \mapsto M(x(t))$ is also locally absolutely continuous. We also observe that, without loss of generality, we shall assume $v =0$. The general case can then be retrieved with a suitable shift \cite[Subsection 2.3]{botkhoa2023}. Following the standard approach, we divide our analysis into two main sections. First, we study the functional $\varphi\colon [t_0, +\infty)\to \R$ defined by
\begin{equation}\label{eq:def_phi}
	\varphi(t) := \frac{1}{2}\|x(t) - x^*\|^2 \quad \text{for} \ t\geq t_0 \ \mbox{and} \ x^* \in \zer M\,.
\end{equation}
Using classical techniques, for example, \cite{cps08}, we establish in Lemma \ref{lem:continuous_phi} a fundamental result that allows us to prove that, whenever $x^*$ is the minimum norm solution of \eqref{eq:intro_monotone_inclusion}, then $\varphi(t)\to 0$ for $t\to +\infty$, provided that, for example, $M(x(t))\to 0$ as $t\to +\infty$. Specifically, this will be a consequence of the vanishing (with a convergence rate) of the second main functional, we study in this paper, $\psi\colon [t_0,+\infty)\to \R$ defined by
\begin{equation}\label{eq:def_psi}
	\psi(t):= \frac{1}{2}\|M(x(t)) + \varepsilon(t)x(t)\|^2 \quad \text{for} \ t \geq t_0\,.
\end{equation}
In particular, following \cite{bk24}, we establish in Lemma \ref{lem:continuous_psi} a fundamental inequality that allows us to obtain rates for $\psi$ expressed in terms of the regularization parameter function $\varepsilon$, including $\psi(t)=O(t^{-2})$ as $t \to +\infty$, for $\varepsilon(t)=\frac{\alpha}{t}$ for all $t \geq t_0$, and $\alpha > 1$.

In the following result, we reproduce the argument presented, e.g., in \cite{cps08}:

\begin{lemma}[Proposition 6 in \cite{cps08}]\label{lem:continuous_phi}
	 Let $x\colon [t_0, +\infty)\to H$ be a solution trajectory to \eqref{eq:tychonov_flow} and $x^*\in \zer M$. Then,
	\begin{equation}\label{eq:ineqlemma1}
		\dot \varphi(t) + \varepsilon(t) \varphi(t) \leq \varepsilon(t)\frac{1}{2} \Big(\|x^*\|^2 -\|x(t)\|^2 \Big)\, \quad \text{for a.e.} \ t \geq t_0\,.
	\end{equation}
\end{lemma}
\begin{proof}
Differentiating $\varphi$ in time we get for a.e.~$t \geq t_0$
\begin{equation}
	\begin{aligned}
		\dot \varphi(t) &= \bra  \dot x(t), x(t) - x^* \ket  = -\bra  M(x(t)) +\varepsilon(t) x(t), x(t) - x^* \ket \\
		& \leq -\varepsilon(t)\bra  x(t), x(t)-x^* \ket  = \varepsilon(t)\frac{1}{2}\Big(\|x^*\|^2 - \|x(t)\|^2\Big)  - \varepsilon(t) \varphi(t)\,,   
	\end{aligned}
\end{equation}
which yields the claim.
\end{proof}

The second result studies the functional $\psi$ defined in \eqref{eq:def_psi} with the adaptation proposed in \cite{bk24}. 

\begin{lemma}[Theorem 2.5 in \cite{bk24}]\label{lem:continuous_psi}
	 Let $x\colon [t_0, +\infty)\to H$ be a solution trajectory to \eqref{eq:tychonov_flow}. Suppose that $t\mapsto x(t)$ is uniformly bounded on $[t_0, +\infty)$. Then, for each $\delta \in (0, 1)$, there exists $C_\delta\geq 0$ such that
	\begin{equation}\label{eq:cont_psi_descent_claim}
		\dot \psi(t) + 2(1-\delta)\varepsilon(t) \psi(t) \leq C_\delta \varepsilon(t) 
 \left(\frac{\dot\varepsilon(t)}{\varepsilon(t)}\right)^2 \, \quad \text{for a.e.} \ t \geq t_0\,.
	\end{equation}
\end{lemma}
\begin{proof}
	Observe first that $\psi(t) = \frac{1}{2}\|\dot x(t)\|^2$ for a.e.~$t \geq t_0$. Take the time derivative of $\psi$ to get for a.e.~$t \geq t_0$
	\begin{equation}\label{eq:cont_descent_psi_1}
		\begin{aligned}
			\dot \psi(t) & = \bra \ddot x(t), \dot x(t) \ket  = -\bra  \frac{d}{dt}M(x(t)) + \varepsilon(t) \dot x(t) + \dot \varepsilon(t) x(t) , \dot x(t) \ket  \\
			& =  -\bra  \frac{d}{dt}M(x(t)), \dot x(t) \ket   - 2 \varepsilon(t)\psi(t) + \dot \varepsilon(t) \bra x(t), \dot x(t)\ket\,.
		\end{aligned}
	\end{equation}
        We now apply the Young inequality on the last term in the right hand-side of \eqref{eq:cont_descent_psi_1} to bound it as follows
        \begin{equation}\label{eq:young_inequality_xx}
            \dot \varepsilon(t) \bra x(t), \dot x(t) \ket  \leq \varepsilon(t) \left(\frac{\dot\varepsilon(t)}{\varepsilon(t)}\right)^2\frac{1}{4\delta}\|x(t)\|^2 + 2\delta \varepsilon(t) \psi(t)\, \quad \text{for a.e.} \ t \geq t_0\,.
        \end{equation}
  To conclude, we use the monotonicity of $M$, which gives $\bra \frac{d}{dt}M(x(t)), \dot x(t)\ket \geq 0$ for a.e.~$t \geq t_0$, and the boundedness of the trajectory, to obtain \eqref{eq:cont_psi_descent_claim} with $C_\delta :=\frac{1}{4\delta} \sup_{t \geq t_0}{\|x(t)\|^2} \geq 0$.
\end{proof}

These two results together with the Gronwall inequality lead us to the following statement.

\begin{theorem}\label{thm:convergence_cont} Let $x$ be a solution trajectory to \eqref{eq:tychonov_flow} with $\varepsilon\colon [t_0, +\infty)\to (0,+\infty)$ continuously differentiable such that
	\begin{equation}\label{eq:continuous_epsilon}
		\lim_{t\to +\infty}\varepsilon(t) = 0\,, \quad \int_{t_0}^{+\infty}\varepsilon(s)ds = +\infty\,, \quad  \text{and}  \quad \displaystyle \lim_{t \to +\infty}\frac{|\dot{\varepsilon}(t)|}{\varepsilon(t)} = 0\,.
	\end{equation}
	Then, $x(t)$ converges strongly to the minimum norm solution to \eqref{eq:intro_monotone_inclusion} as $t \rightarrow +\infty$, and the convergence rate of the residual results from
	\begin{equation}\label{eq:rate_psi_abstract}
		\psi(t)\leq \frac{ \psi(t_0)}{\gamma(t) } + \frac{C_\delta}{\gamma(t) } \int_{t_0}^t \gamma(s) \varepsilon(s) \left(\frac{\dot\varepsilon(t)}{\varepsilon(t)}\right)^2 ds\, \quad \text{for all} \  t\geq t_0\,,
	\end{equation}
    and any $\delta \in (0,1)$, where $\gamma(t):= \exp\Big( 2(1-\delta) \int_{t_0}^{t}\varepsilon(s)ds\Big)$ and $C_\delta:=\frac{1}{4\delta} \sup_{t \geq t_0}{\|x(t)\|^2} \geq 0$.
\end{theorem}
\begin{proof}
From Lemma \ref{lem:continuous_phi} and Lemma \ref{lem:limsup_continuous} we deduce that $t \mapsto x(t)$ is uniformly bounded on $[t_0, +\infty)$. Thus, Lemma \ref{lem:continuous_psi} applies, and, coupled with Lemma \ref{lem:limsup_continuous}, yields \eqref{eq:rate_psi_abstract}. Owing to the hypotheses in \eqref{eq:continuous_epsilon} and Lemma \ref{lem:limsup_continuous} this in turn implies $M(x(t))\to 0$ as $t\to +\infty$. This shows that all cluster points of $x(t)$ as $t\to +\infty$ lie in $\zer M$, due to the weak-to-strong closeness of the graph of $M$. Therefore, taking $x^*$ as the minimum norm solution, using the weak lower semicontinuity of the norm, we get, for $h(t):= \|x^*\|^2 - \|x(t)\|^2$, $\limsup_{t\to+\infty} h(t) = 0$. Therefore, applying Lemma \ref{lem:limsup_continuous} again, but this time for \eqref{eq:ineqlemma1}, we get $x(t)\to x^*$ strongly as $t \to +\infty$.
\end{proof}

The rate reduces to the expected $O(t^{-1})$ rate for choices of $\varepsilon$ such as $\varepsilon(t)=\frac{\alpha}{t}$ with $\alpha > 1$.

\begin{corollary}\label{cor:continuous_optimal_rate}
    Let $x$ be a solution trajectory to \eqref{eq:tychonov_flow} with $\varepsilon\colon [t_0, +\infty)\to (0,+\infty)$ defined by $\varepsilon(t)=\frac{\alpha}{t}$ and $\alpha > 0$. Then, $x(t)$ converges strongly to the minimum norm solution to \eqref{eq:intro_monotone_inclusion} as $t \rightarrow +\infty$, the residual showing the following convergence rate:
    \begin{enumerate}
        \item If $\alpha > 1$, then $\|M(x(t))\|=O(t^{-1})$ as $t \to +\infty$;\label{item:rate_alpha_1}
        \item If $\alpha \leq 1$, then $\|M(x(t))\|=O(t^{-\alpha(1-\delta)})$ for any $\delta \in (0, 1)$, as $t \to + \infty$.\label{item:rate_alpha_2}
    \end{enumerate}
\end{corollary}
\begin{proof}
	Let $\delta \in (0, 1)$ and consider the abstract rate \eqref{eq:rate_psi_abstract}. A simple calculation shows that for all $t \geq t_0$, we have $\gamma(t)=C t^{2\alpha(1-\delta)}$, for $C >0$, and $\left(\frac{\dot \varepsilon(s)}{\varepsilon(s)}\right)^2 = \frac{1}{s^2}$ for all $s \in [t_0, t]$ . Hence, \eqref{eq:rate_psi_abstract} turns into
    \begin{equation}
        \psi(t) = O\left( t^{-2 \alpha (1-\delta)}   + t^{-2}\right) \quad \text{as} \ t \to +\infty\,.
    \end{equation}
    Suppose that $\alpha > 1$. Then for $\delta \in (0,1)$ small enough $2 \alpha (1-\delta) > 2$, which yields $\psi(t) = O(t^{-2})$ as $t \to + \infty$. Now, since $x$ is bounded, $(2\psi(t))^\frac{1}{2} + O(t^{-1}) \geq \|M(x(t))\|$ as $t \to + \infty$, hence the item \ref{item:rate_alpha_1}. Using the same argument, if $\alpha \leq 1$ we obtain $\psi(t)=O(t^{-2\alpha(1-\delta)})$ as $t \to +\infty$ and hence $\|M(x(t))\| = O(t^{-\alpha (1-\delta)})$ as $t \to +\infty$.
\end{proof}

\begin{remark}[The case $\alpha = 1$]\label{rmk:alpha_one}
 The case $\alpha = 1$ can also be shown to give the rate $O(t^{-1})$ as $t\to +\infty$. This has been studied in \cite{yr21} with a somewhat more involved analysis and in \cite{bk24} in the context of a similar approach with general parameters, but requiring the additional assumption that $t\mapsto \frac{d}{dt}(\dot \varepsilon(t) + \varepsilon^2(t))$ preserves constant sign on $[t_0, +\infty)$. The latter follows naturally with a more careful analysis, namely, instead of applying the Young inequality to \eqref{eq:young_inequality_xx}, one keeps the term
 \begin{equation}
     \dot \varepsilon(t) \bra x(t), \dot x(t)\ket = \dot \varepsilon(t)\frac{d}{dt} \frac{1}{2}\|x(t)\|^2\,,
 \end{equation}
applies the Gronwall inequality as in the proof of Theorem \ref{thm:convergence_cont} and integrates by parts. However, this approach seems particularly difficult to reproduce in discrete time, so we decided to go for the simpler argument. The corollary \ref{cor:continuous_optimal_rate} thus shows that the asymptotic rate, at least for $\alpha > 1$, can be obtained even by removing the sign assumption considered in \cite{bk24}, but producing a suboptimal result for $\alpha = 1$. Note that this comes at an important price --- the constant $C_\delta$ in \eqref{eq:rate_psi_abstract} explodes as $\delta \to 0$.
\end{remark}

\section{A General Framework for the Extra-Anchored-Gradient method}\label{sec:gfeg}

In this section, we study the asymptotic behavior of the Algorithm \ref{alg:gfeg}. First, let us observe once again that without loss of generality, we shall assume $v =0$. The general case can then be retrieved with an appropriate shift. The iterative scheme produces for a starting point $x^0 \in H$ two sequences $(x^k)_k$ and $(y^{k+1})_k$ defined for all $k \geq 0$ by
\begin{equation}\label{eq:anchored_eag}
	\left\{\begin{aligned}
		&y^{k+1} := (1- \theta \varepsilon^k) x^k - \theta M(x^k)\,,\\
		&x^{k+1} =x^k - \theta \varepsilon^{k+1}x^{k+1} - \theta M(y^{k+1})\,,
	\end{aligned}\right.
\end{equation}
where $(\varepsilon^k)_k$ is a non-negative sequence of parameters satisfying:	\begin{equation}\label{eq:condition_for_eps}
\lim_{k\to +\infty} \varepsilon^k = 0\,, \quad \sum_{k=0}^{+\infty} \varepsilon^k = +\infty\,, \quad \text{and} \quad \lim_{k\to +\infty} \frac{|\varepsilon^{k+1}-\varepsilon^k|}{\varepsilon^k} = 0\,.
\end{equation}
The second identity in \eqref{eq:anchored_eag} can be written as
\begin{equation}\label{eq:update_for_x}
    \frac{\Delta x_k}{\theta}  + M(y^{k+1}) + \varepsilon^{k+1}x^{k+1} = 0 \quad \text{for all} \ k \geq 0\,,
\end{equation}
which shows that \eqref{eq:anchored_eag} can be understood as an explicit discretization of \eqref{eq:tychonov_flow}, $y^{k+1}$ being an extrapolation point. If $\varepsilon^k=0$ for all $k \geq 0$, the method reduces to EG \cite{korpelevich1976}. Note that, unlike Popov-type schemes, the variable $y^{k+1}$ is in fact only temporary, and thus the method only requires storing one variable $x^k$. Using the terminology of \cite{Ryu}, we shall say that the method has \emph{minimal lifting} and is not \emph{frugal}.  

Following the continuous time setting, we divide the analysis into two main parts. This time, we start with the analysis of the sequence $(\psi^k)_k$ defined by
\begin{equation}\label{eq:def_psi_discrete}
\psi^k:= \frac{1}{2}\|M(x^k) + \varepsilon^k x^k\|^2\, \quad \text{for all} \ k \geq 0\,,
\end{equation}
for which we provide in Lemma \ref{lem:rate_explicit} below an analogous result to the continuous time setting in Lemma \ref{lem:continuous_psi}. Then, we study the sequence $(\varphi^k)_k$ defined by
\begin{equation}\label{eq:def_phi_discrete}
	\varphi^k:= \frac{1}{2}\|x^k -x^*\|^2\, \quad \text{for all} \ k \geq 0 \ \text{and} \ x^* \in \zer M\,,
\end{equation}
for which we provide in Lemma \ref{lem:strong_convergence_explicit} below the discrete time analogous of Lemma \ref{lem:continuous_phi}. These two preliminary results are collected in Section \ref{sec:technical_results}.

\subsection{Technical Lemmas}\label{sec:technical_results}

To strengthen the connection with the continuous time approach and make the rationale of the proof of Lemma \ref{lem:rate_explicit} below easier to follow, we introduce two important simplifications. First, leveraging the analogy between finite differences and derivatives, given a sequence $(a^k)_k$ we denote $\Delta a_k := a^{k+1}-a^k$ for all $k \geq 0$. In this sense, it will be useful to recall some discrete calculus rules listed in Appendix \ref{sec:discrete_calculus}. Secondly, we introduce an auxiliary sequence $(z^k)_k$ defined by
\begin{equation}
    z^k :=M(x^k) + \varepsilon^k x^k \quad \text{for all} \ k \geq 0\,. 
\end{equation}
The sequences $(x^k)_k$, $(y^{k+1})_k$, and $(z^k)_k$ are related by the following fundamental relations, which, although unmotivated at first sight, arise naturally from our Lyapunov analysis
\begin{align}
	\label{eq:def_y} &\Delta x_k + \theta z^k = -\theta \Delta (\varepsilon x)_k + \theta^2 \varepsilon^k z^k + \theta N^k\,,\\
	\label{eq:def_m} \mbox{where} \quad & N^k := M(x^k) - \theta \varepsilon^k z^k - M(y^{k+1});\\
	\label{eq:def_xi} & y^{k+1} = x^k -\theta z^k\,.
\end{align}

\begin{lemma}[$\psi$-Analysis]\label{lem:rate_explicit}
    Let $(x^k)_k$ and $(y^{k+1})_k$ be generated by \eqref{eq:anchored_eag} with $\theta \in \big(0, \frac{1}{L}\big)$ and $(\varepsilon^k)_k$ satisfying \eqref{eq:condition_for_eps}. Suppose that $(x^k)_k$ is bounded. Then, for any $\delta \in (0,1)$, there exist $C_\delta \geq 0$ and $k_0\geq 0$ such that
    \begin{equation}\label{eq:decrease_psi_discrete}
    	\begin{aligned}
    		\Delta \psi_k + 2(1 -\delta)\theta\varepsilon^k \psi^k \leq  C_\delta \theta \varepsilon^k \left(\frac{\Delta \varepsilon_k}{\theta \varepsilon^k}\right)^2\, \quad \text{for all} \ k \geq k_0\,.
     	\end{aligned}
    \end{equation} 
\end{lemma}
\begin{proof} Let $\delta \in (0,1)$. To shorten notation, we write $M^k:=M(x^k)$ for all $k \geq 0$. Recall that $z^k = M^k + \varepsilon^k x^k$ and hence $\psi^k = \frac{1}{2}\|z^k\|^2$. Differentiating, we get for all $k \geq 0$
    \begin{equation}
        \begin{aligned}
            \Delta \psi_k & = \bra \Delta M_{k} + \varepsilon^k \Delta x_k + \Delta \varepsilon_k x^{k+1}, z^k\ket + \frac{1}{2}\| \Delta z_k\|^2\\
            & = \bra \Delta M_{k}, z^k\ket - 2\varepsilon^k \theta \psi^k + \varepsilon^k \bra\Delta x_k + \theta z^k, z^k\ket + \Delta \varepsilon_k\bra  x^{k+1}, z^k\ket + \frac{1}{2}\| \Delta z_k\|^2\\
             & = -\frac{1}{\theta} \bra \Delta M_{k}, \Delta x_k \ket - 2\varepsilon^k \theta \psi^k  + \Delta \varepsilon_k\bra  x^{k+1}, z^k\ket\\
             & \quad +\underbrace{\frac{1}{\theta} \bra \Delta M_{k}, \Delta x_k + \theta z^k \ket + \varepsilon^k \bra\Delta x_k + \theta z^k, z^k\ket + \frac{1}{2}\| \Delta z_k\|^2}_{(\labterm{label:psi_I}{I})}\,.
        \end{aligned}
    \end{equation}
    We study \eqref{label:psi_I} separately. For $N^k$ satisfying the identity \eqref{eq:def_y}, we have for all $k \geq 0$ the following
    \begin{equation}\label{eq:psi_I_1}
        \begin{aligned}
            \eqref{label:psi_I} \ & := \frac{1}{\theta} \bra \Delta M_{k}, \Delta x_k + \theta z^k \ket + \varepsilon^k \bra\Delta x_k + \theta z^k, z^k\ket + \frac{1}{2}\| \Delta z_k\|^2\\
            & = \frac{1}{\theta} \bra \Delta M_{k} + \theta \varepsilon^k z^k, \Delta x_k + \theta z^k \ket  + \frac{1}{2}\| \Delta z_k\|^2\\
        & = -\bra \Delta M_{k} + \theta \varepsilon^k z^k,  \Delta (\varepsilon x)_k - \theta \varepsilon^k z^k - N^k \ket  + \frac{1}{2}\| \Delta z_k\|^2\\
        & = -\bra \Delta M_{k} + \theta \varepsilon^k z^k,  \Delta (\varepsilon x)_k - \theta \varepsilon^k z^k\ket  + \frac{1}{2}\| \Delta z_k\|^2 + \bra \Delta M_k + \theta \varepsilon^kz^k , N^k \ket\\
         & =\underbrace{ \frac{1}{2}\|\Delta M_{k} + \theta \varepsilon^k z^k\|^2 + \bra \Delta M_k + \theta \varepsilon^k z^k, N^k \ket}_{(\labterm{label:psi_I_1}{I.1})} + \underbrace{\frac{1}{2}\| \Delta (\varepsilon x)_k - \theta \varepsilon^k z^k\|^2}_{(\labterm{label:psi_N1}{N.1})}\,,
    \end{aligned}
    \end{equation}
    where in the last line we have employed the polarization identity on the first inner product and the fact that $\Delta z_k = \Delta M_k + \Delta (\varepsilon x)_k$. Now we show that \eqref{label:psi_I_1} can be bounded from above picking $N^k$ as in \eqref{eq:def_m} and $y^{k+1}$ satisfying \eqref{eq:def_xi}. For all $k \geq 0$ we have 
    \begin{equation}
        \begin{aligned}
            \eqref{label:psi_I_1}&:=  \frac{1}{2}\|\Delta M_{k} + \theta \varepsilon^k z^k\|^2 + \bra \Delta M_k + \theta \varepsilon^k z^k, N^k \ket\\
            & =  \frac{1}{2}\|\Delta M_{k} + \theta \varepsilon^k z^k + N^k\|^2 - \frac{1}{2}\|N^k\|^2\\
            & = \frac{1}{2}\|M^{k+1} - M(y^{k+1})\|^2 - \frac{1}{2}\|N^k\|^2\\
            & \leq \frac{L^2}{2}\|x^{k+1}-y^{k+1}\|^2 - \frac{1}{2}\| N^k\|^2\\
            & = \frac{L^2}{2}\|x^k - \theta z^k -\theta \Delta (\varepsilon x)_k + \theta^2 \varepsilon^k z^k + \theta N^k- y^{k+1}\|^2 - \frac{1}{2}\|N^k\|^2\\
			& = \frac{L^2}{2}\|\theta N^k-\theta \Delta (\varepsilon x)_k + \theta^2 \varepsilon^k z^k\|^2 - \frac{1}{2}\|N^k\|^2\\
			& \leq \frac{1}{2}\left( (\theta L)^2(1+\delta_1) - 1\right)\| N^k\|^2 + \underbrace{\frac{L^2}{2}\left(1+ \frac{1}{\delta_1}\right)\|\theta \Delta (\varepsilon x)_k - \theta^2 \varepsilon^k z^k\|^2}_{(\labterm{label:psi_N2}{N.2})} \leq \eqref{label:psi_N2}\,,
		\end{aligned}
	\end{equation}
	for some $\delta_1 >0$ such that $1+\delta_1 < \frac{1}{(\theta L)^2}$, which exists since $\theta < \frac{1}{L}$. Let us denote by $(\labterm{label:psi_N}{N}) := \eqref{label:psi_N1} + \eqref{label:psi_N2}$. Putting everything together, using the monotonicity of $M$, and the Young inequality, we get for all $k \geq 0$
	\begin{equation}\label{eq:descent_psi_2}
		\begin{aligned}
			\Delta \psi_k  \leq&  - 2 \theta \varepsilon^k  \psi^k  + \Delta \varepsilon_k \bra  x^{k+1}, z^k \ket  + \eqref{label:psi_N}\\
			 \leq & -2 \theta \varepsilon^k  \psi^k  + \theta \varepsilon^k  \left(\frac{\Delta \varepsilon_k}{\theta \varepsilon^k}\right)^2 \frac{1}{2 \delta}\|x^{k+1}\|^2  + \theta \varepsilon^k  \frac{\delta}{2} \|z^k\|^2 + \eqref{label:psi_N}\\
            \leq & -(2 - \delta ) \theta \varepsilon^k  \psi^k  + \theta \varepsilon^k  C_1 \left(\frac{\Delta \varepsilon_k}{\theta \varepsilon^k}\right)^2 + \eqref{label:psi_N}\,.
		\end{aligned}
	\end{equation}
    where $C_1 := \frac{1}{2\delta} \sup_{k \geq 0} \|x^{k+1}\|^2$. Let us now deal with \eqref{label:psi_N}. Let $C_2 :=   1+ \theta^2L^2 + \frac{\theta^2L^2}{\delta_1}$ and set $k_0'>0$ such that $C_2 \theta \varepsilon^k \leq \frac{\delta}{4}$ for all $k \geq k_0'$, which exists since $\varepsilon^k\to 0$ as $k \to +\infty$ from \eqref{eq:condition_for_eps}. For all $k \geq k_0'$, we have:
    \begin{equation}
        \begin{aligned}
            \eqref{label:psi_N} & = \frac{1}{2} \left(1+ \theta^2L^2 + \frac{\theta^2L^2}{\delta_1} \right) \| \Delta (\varepsilon x)_k - \theta \varepsilon^k z^k\|^2\\
             & \leq  C_2 \| \Delta (\varepsilon x)_k\|^2 +  C_2 \theta^2 (\varepsilon^k)^2\|z^k\|^2 \leq C_2 \| \Delta (\varepsilon x)_k\|^2 +  \theta \varepsilon^k \frac{\delta}{2} \psi^k\,.
        \end{aligned}
    \end{equation}
    Observe now that $\Delta (\varepsilon x)_k = \Delta \varepsilon_k x^{k+1} + \varepsilon^k \Delta x_{k}$ and
    \begin{equation}\label{eq:control_deltay}
        \begin{aligned}
            \Delta x_{k} &= -\theta \varepsilon^{k+1} x^{k+1} - \theta M(y^{k+1})\\
            & = -\theta \varepsilon^{k+1} x^{k+1} - \theta (M(y^{k+1}) - M^k)- \theta M^k - \theta \varepsilon^k x^k +\theta \varepsilon^k x^k\\
            & = -\theta \Delta \varepsilon_k x^{k+1} - \theta (M(y^{k+1}) - M^k) - \theta z^k - \theta \varepsilon^k \Delta x_k\,,
        \end{aligned}
    \end{equation}
    hence, isolating $\Delta x_k$, we get for all $k \geq 0$
    \begin{equation}
        \Delta x_k =  \frac{1}{1 + \theta \varepsilon^k}\left(-\theta \Delta \varepsilon_k x^{k+1} - \theta (M(y^{k+1}) - M^k) - \theta z^k \right).
    \end{equation}
    Therefore, since $y^{k+1} - x^k = -\theta z^k$ from \eqref{eq:anchored_eag} and by using the Lipschitz property of $M$ and \eqref{eq:control_deltay}, there exist $C_3 >0$ and $k_0'' \geq 0$ such that for all $k \geq k_0''$
    \begin{equation}\label{eq:control_delta_ex}
        \begin{aligned}
           \eqref{label:psi_N} & \leq C_2 \|\Delta (\varepsilon x)_k\|^2 + \theta \varepsilon^k \frac{\delta}{2}\psi^k\\
           & = \frac{C_2}{(1 + \theta \varepsilon^k)^2} \|\Delta \varepsilon_k x^{k+1} -\theta \varepsilon^k(M(y^{k+1}) - M^k) - \theta \varepsilon^k z^k\|^2 + \theta \varepsilon^k \frac{\delta}{2}\psi^k\\
            & \leq \frac{3 C_2 (\Delta \varepsilon_k)^2}{(1 + \theta \varepsilon^k)^2} \|x^{k+1}\|^2 + \frac{3 C_2 (\theta \varepsilon_k)^2}{(1 + \theta \varepsilon^k)^2}(\theta^2L^2 + 1)\|z^k\|^2 + \theta \varepsilon^k \frac{\delta}{2}\psi^k\\
            & \leq 3 C_2 (\theta \varepsilon^k)^2 \left(\frac{\Delta \varepsilon_k}{\theta \varepsilon^k}\right)^2 \|x^{k+1}\|^2 + \frac{6 C_2 (\theta \varepsilon_k)^2 }{(1 + \theta \varepsilon^k)^2}(\theta^2L^2 + 1)\psi^k + \theta \varepsilon^k \frac{\delta}{2}\psi^k\\
            & \leq \theta \varepsilon^k C_3 \left(\frac{\Delta \varepsilon_k}{\theta \varepsilon^k}\right)^2 + \theta \varepsilon^k \delta \psi^k.
        \end{aligned}
    \end{equation}
The conclusion follows by plugging \eqref{eq:control_delta_ex} into \eqref{eq:descent_psi_2}, for $k_0:=\max\{k_0', k_0''\}$ and $C_\delta:=C_1 + C_3$.
\end{proof}

\begin{remark}[Rate for EG]\label{rmk:eg}
Inspecting the proof of Lemma \ref{lem:rate_explicit} we can notice that assuming $\varepsilon^k:=0$ for all $k \geq 0$, we have obtained a proof for the last-iterate convergence rate for EG. In fact, setting $\varepsilon^k=0$ and following the same proof up to the first line in \eqref{eq:descent_psi_2} we get, actually for $\theta \in \big(0, \frac{1}{L}\big]$,
\begin{equation}
\Delta \psi_k \leq 0\,, \quad \text{i.e.,} \quad \|M(x^{k+1})\| \leq \|M(x^k)\|\, \quad \text{for all} \ k \geq 0. 
\end{equation}
If $\theta < \frac{1}{L}$, this descent property of $(\|M(x^k)\|)_k$ coupled with the fact that $(\|M(x^k)\|)_k \in \ell^2$ (cf.~(55) in \cite{glg22}), yields $\|M(x^k)\|=o(k^{-{\frac{1}{2}}})$ as $k \to +\infty$. This result was established in \cite[Lemma 3.2]{glg22} only for $\theta \in \big(0, \frac{1}{\sqrt{2}L}\big]$ with a computer-assisted proof. 
\end{remark}

We now develop the analogous result of Lemma \ref{lem:continuous_phi}.

\begin{lemma}[$\varphi$-analysis]\label{lem:strong_convergence_explicit}
	 Let $(x^k)_k$ and $(y^{k+1})_k$ be generated by \eqref{eq:anchored_eag} with $\theta \in \big(0, \frac{1}{L}\big)$ and $(\varepsilon^k)_k$ satisfying \eqref{eq:condition_for_eps}. Let $(\varphi^k)_k$ defined as in \eqref{eq:def_phi_discrete} for some $x^* \in \zer M$. Then, there exist a vanishing real sequence $(\eta^k)_k$ and $k_1\geq 0$, such that
	\begin{equation}
		\begin{aligned}
			\Delta \varphi_{k} + \frac{\theta \varepsilon^{k+1}}{2} \varphi^{k} \leq    \frac{\theta \varepsilon^{k+1}}{2}\left(\|x^*\|^2 - \frac{1 - \eta^{k}}{1+\theta \varepsilon^{k+1}}\|x^{k}\|^2\right)\, \quad \text{for all} \ k \geq k_1\,.
		\end{aligned}
	\end{equation}
\end{lemma}
\begin{proof}
	Let $k \geq 0$. We have
	\begin{equation}
		\begin{aligned}
			\Delta \varphi_{k} & = \bra \Delta x_{k}, x^{k+1} - x^*  \ket  - \frac{1}{2}\|\Delta x_{k}\|^2\\
			& =  \bra  -\theta M(y^{k+1}) - \theta \varepsilon^{k+1}x^{k+1}, x^{k+1} - x^*  \ket  - \frac{1}{2}\| \Delta x_{k}\|^2 \\
			& = -\theta  \bra  M(y^{k+1}), x^{k+1} - x^* \ket  - \theta \varepsilon^{k+1}  \bra  x^{k+1}, x^{k+1} - x^* \ket  - \frac{1}{2}\|\Delta x_{k}\|^2\\
			& = -\theta  \bra  M(y^{k+1}), y^{k+1} - x^* \ket  - \theta \varepsilon^{k+1}  \bra  x^{k+1}, x^{k+1} - x^* \ket\\
                & \quad \underbrace{- \theta   \bra  M(y^{k+1}), x^{k+1} - y^{k+1} \ket - \frac{1}{2}\|\Delta x_{k}\|^2}_{(\labterm{label:eag_I}{I})}\,.\\
		\end{aligned}
	\end{equation}
	We study \eqref{label:eag_I} separately
	\begin{equation}
		\begin{aligned}
			\eqref{label:eag_I}:=& - \theta   \bra  M(y^{k+1}), x^{k+1} - y^{k+1} \ket - \frac{1}{2}\|\Delta x_{k}\|^2\\
			= &  \underbrace{- \theta   \bra  M(y^{k+1}), x^{k} - y^{k+1} \ket }_{(\labterm{label:eag_I.1}{I.1})}- \frac{1}{2}\|\Delta x_{k}\|^2 - \theta  \bra  M(y^{k+1}), \Delta x_{k} \ket \,.
		\end{aligned}
	\end{equation}
	Regarding the term \eqref{label:eag_I.1}, using the update formula for $y^{k+1}$ (see \eqref{eq:anchored_eag}), the polarization identity and the update formula for $x^{k+1}$ (see \eqref{eq:update_for_x}), we get
	\begin{equation}
		\begin{aligned}
			\eqref{label:eag_I.1}& :=  - \theta   \bra  M(y^{k+1}), x^{k} - y^{k+1} \ket \\
			& =  - \theta   \bra  M(y^{k+1}), \theta \varepsilon^{k} x^{k} + \theta M(x^{k}) \ket \\
			& =  -  \bra  \theta M(y^{k+1}), \theta \varepsilon^{k+1} x^{k+1} \ket  - \theta   \bra  M(y^{k+1}), \theta \varepsilon^{k} x^{k} - \theta \varepsilon^{k+1} x^{k+1} \ket \\
                & \quad - \theta^2 \bra  M(y^{k+1}),  M(x^{k}) \ket \\
			& = -\frac{1}{2}\|\Delta x_{k}\|^2 + \frac{1}{2}\|\theta \varepsilon^{k+1} x^{k+1} \|^2 + \theta^2  \bra  M(y^{k+1}),  \Delta(\varepsilon x)_{k} \ket  \\
			& \quad   + \underbrace{\frac{\theta^2}{2}\|M(y^{k+1})\|^2 - \theta^2  \bra M(y^{k+1}), M(x^{k}) \ket }_{(\labterm{label:eag_I.2}{I.2})}\,.
		\end{aligned}
	\end{equation}
	To bound \eqref{label:eag_I.2}, we make use of the Lipschitz continuity of $M$, the update formula for $y^{k+1}$ (see \eqref{eq:anchored_eag}) and the Young inequality as follows
	\begin{equation}
		\begin{aligned}
			 \eqref{label:eag_I.2}&:=\frac{\theta^2}{2}\|M(y^{k+1})\|^2 - \theta^2  \bra  M(y^{k+1}), M(x^{k}) \ket \\
			& = \frac{\theta^2}{2}\|M(y^{k+1}) - M(x^{k})\|^2 - \frac{\theta^2}{2} \|M(x^{k})\|^2 \\
			& \leq \frac{(\theta L)^2}{2}\| \theta \varepsilon^{k} x^{k} + \theta M(x^{k})\|^2 - \frac{\theta^2}{2} \|M(x^{k})\|^2 \\
			& \leq \frac{(\theta L)^2}{2}\left(1 + \frac{1}{\delta_1}\right)\| \theta \varepsilon^{k} x^{k}\|^2 - \frac{\theta^2}{2} \left(1 - (\theta L )^2 (1 + \delta_1) \right)  \|M(x^{k})\|^2 \\
             & \leq \frac{(\theta L)^2}{2}\left(1 + \frac{1}{\delta_1}\right)\| \theta \varepsilon^{k} x^{k}\|^2,
		\end{aligned}
	\end{equation}
	for some $\delta_1 >0$ such that $1+\delta_1 < \frac{1}{(\theta L)^2}$, which exists since $\theta < \frac{1}{L}$. Plugging everything together and using the monotonicity of $M$, we get
	\begin{equation}\label{eq:dpsi_1}
		\begin{aligned}
			\Delta \varphi_{k} & \leq  - \theta \varepsilon^{k+1}  \bra  x^{k+1}, x^{k+1} - x^* \ket  &\\
			& \quad + \frac{(\theta L)^2}{2}\left(1 + \frac{1}{\delta_1}\right) \| \theta \varepsilon^{k} x^{k}\|^2 & \text{$\rhd$ \textcolor{gray}{Remainder of \eqref{label:eag_I.2}}}\\
			& \quad  -\frac{1}{2}\|\Delta x_{k}\|^2 + \frac{1}{2}\|\theta \varepsilon^{k+1} x^{k+1}\|^2 + \theta^2  \bra  M(y^{k+1}),  \Delta(\varepsilon x)_{k} \ket  & \text{$\rhd$ \textcolor{gray}{Remainder of \eqref{label:eag_I.1}}}\\
			& \quad - \frac{1}{2}\|\Delta x_{k}\|^2 - \theta  \bra  M(y^{k+1}), \Delta x_{k} \ket \,. & \text{$\rhd$ \textcolor{gray}{Remainder of \eqref{label:eag_I}}}\hspace{0.281cm}
		\end{aligned}
	\end{equation}
	The first term in \eqref{eq:dpsi_1} can be written as in the continuous time approach as
	\begin{equation}\label{eq:parallel}
		- \theta \varepsilon^{k+1}  \bra  x^{k+1}, x^{k+1} - x^* \ket  = \frac{\theta \varepsilon^{k+1}}{2}\left(\|x^*\|^2 - \|x^{k+1}\|^2\right) - \theta \varepsilon^{k+1} \varphi^{k+1}\,.
	\end{equation}
	We now deal with the remainder in \eqref{eq:dpsi_1}, where we denote by $C:= \frac{(\theta L)^2}{2}\left(1 + \frac{1}{\delta_1}\right)$, i.e.,
	\begin{equation}\label{eq:remainder_phi}
		\begin{aligned}
			&- \|\Delta x_{k}\|^2 + \underbrace{\frac{C(\theta\varepsilon^{k})^2}{2}\| x^{k} \|^2 + \frac{1}{2}\|\theta \varepsilon^{k+1} x^{k+1}\|^2}_{(\labterm{label:phi_N1}{N.1})}  +\underbrace{\theta  \bra  M(y^{k+1}), \theta \Delta (\varepsilon x)_{k} - \Delta x_{k}  \ket }_{(\labterm{label:phi_II}{II})}\,.
		\end{aligned}
	\end{equation}
	We start from \eqref{label:phi_II}. Using the definition of $x^{k+1}$ (see \eqref{eq:update_for_x}), we get
	\begin{equation}
		\begin{aligned}
			\eqref{label:phi_II} & := \theta  \bra  M(y^{k+1}), \theta \Delta (\varepsilon x)_{k} - \Delta x_{k}  \ket  =   \bra  \theta M(y^{k+1}), -(1-\theta \varepsilon^{k})\Delta x_{k} + \theta \Delta \varepsilon_{k} x^{k+1}  \ket \\
			& =   \bra  -\Delta x_{k} - \theta \varepsilon^{k+1} x^{k+1}, -(1-\theta \varepsilon^{k})\Delta x_{k} + \theta \Delta \varepsilon_{k} x^{k+1}  \ket \\
			& = (1-\theta \varepsilon^{k})\|\Delta x_{k}\|^2 - \theta \Delta \varepsilon_{k} \bra  \Delta x_{k}, x^{k+1} \ket  \\
			& \quad + \theta\varepsilon^{k+1} (1-\theta \varepsilon^{k}) \bra  x^{k+1}, \Delta x_{k} \ket  - \theta^2 \varepsilon^{k+1} \Delta \varepsilon_{k}\|x^{k+1}\|^2\\
                & = (1-\theta \varepsilon^{k})\|\Delta x_{k}\|^2 + \theta \varepsilon^{k+1} \bra  x^{k+1}, \Delta x_{k} \ket  \\
			& \quad \underbrace{-\theta\left(\Delta \varepsilon_{k} + \theta \varepsilon^{k+1} \varepsilon^{k} \right)\bra  x^{k+1}, \Delta x_{k} \ket  - \theta^2 \varepsilon^{k+1} \Delta \varepsilon_{k}\|x^{k+1}\|^2}_{(\labterm{label:phi_N2}{N.2})}\,.
		\end{aligned}
	\end{equation}
Plugging \eqref{label:phi_II} into \eqref{eq:remainder_phi} and denoting by $(\labterm{label:phi_N}{N}):=\eqref{label:phi_N1} + \eqref{label:phi_N2}$, we get
\begin{equation*}
	\begin{aligned}
		\eqref{eq:remainder_phi}=& - \|\Delta x_{k}\|^2 + (1-\theta \varepsilon^{k})\|\Delta x_{k}\|^2 +  \theta \varepsilon^{k+1} \bra  x^{k+1}, \Delta x_{k} \ket + \eqref{label:phi_N}\\
		=& -\theta \varepsilon^{k} \|\Delta x_{k}\|^2 +  \theta \varepsilon^{k+1}\left(\frac{1}{2}\|x^{k+1}\|^2 + \frac{1}{2}\|\Delta x_{k}\|^2 - \frac{1}{2}\|x^{k}\|^2  \right) + \eqref{label:phi_N}\\
        =& \frac{\theta \varepsilon^{k+1}}{2}\|x^{k+1}\|^2 - \frac{\theta \varepsilon^{k+1}}{2}\|x^{k}\|^2 - \frac{\theta \varepsilon^{k}}{2}\left(1 - \eta_1^{k}\right) \|\Delta x_{k}\|^2 + \eqref{label:phi_N}\,,
	\end{aligned}
 \end{equation*}
where $\eta_1^{k} := \frac{\Delta \varepsilon_{k}}{\varepsilon^{k}}$ is an infinitesimal term as $k \to +\infty$ from \eqref{eq:condition_for_eps}. We deal with \eqref{label:phi_N} as follows
    \begin{equation*}
        \begin{aligned}
            \eqref{label:phi_N}= & \frac{C(\theta\varepsilon^{k})^2}{2}\| x^{k}\|^2 + \frac{1}{2}\|\theta \varepsilon^{k+1} x^{k+1}\|^2 -\theta\left(\Delta \varepsilon_{k} + \theta \varepsilon^{k+1} \varepsilon^{k} \right)\bra  x^{k+1}, \Delta x_{k} \ket  - \theta^2 \varepsilon^{k+1} \Delta \varepsilon_{k}\|x^{k+1}\|^2\\
		= & \left( (\theta \varepsilon^{k+1})^2 -  2\theta^2 \varepsilon^{k+1} \Delta \varepsilon_{k} \right)\frac{1}{2}\|x^{k+1}\|^2 \\
		&-\theta\left(\Delta \varepsilon_{k} + \theta \varepsilon^{k+1} \varepsilon^{k} \right) \left(\|\Delta x_{k}\|^2 + \bra x^{k}, \Delta x_{k}\ket \right) +  \frac{C(\theta \varepsilon^{k})^2}{2}\|x^{k}\|^2\\
		= & \left( (\theta \varepsilon^{k+1})^2 -  2\theta^2 \varepsilon^{k+1} \Delta \varepsilon_{k} \right)\left(\frac{1}{2}\|x^{k}\|^2+ \bra  \Delta x_{k}, x^{k} \ket  + \frac{1}{2}\|\Delta x_{k}\|^2\right)\\
		& -\theta\left(\Delta \varepsilon_{k} + \theta \varepsilon^{k+1} \varepsilon^{k} \right) \left(\|\Delta x_{k}\|^2 + \bra x^{k}, \Delta x_{k}\ket \right) +  \frac{C(\theta \varepsilon^{k})^2}{2}\|x^{k}\|^2\\
         \leq & \left( (\theta \varepsilon^{k+1})^2 -  2\theta^2 \varepsilon^{k+1} \Delta \varepsilon_{k} + C(\theta \varepsilon^{k})^2 \right) \frac{1}{2}\|x^{k}\|^2 \\
         & +  \left( (\theta \varepsilon^{k+1})^2 -  2\theta^2 \varepsilon^{k+1} \Delta \varepsilon_{k} -2\theta \Delta \varepsilon_{k} - 2\theta^2 \varepsilon^{k+1} \varepsilon^{k} \right)  \frac{1}{2}\|\Delta x_{k}\|^2 \\
		& + \left |  (\theta \varepsilon^{k+1})^2 -  2\theta^2 \varepsilon^{k+1} \Delta \varepsilon_{k} - \theta \Delta \varepsilon_{k} - \theta^2 \varepsilon^{k+1} \varepsilon^{k} \right | \left(\frac{1}{2}\|x^{k}\|^2 + \frac{1}{2}\|\Delta x_{k}\|^2 \right)\\ 
   = & \ \eta^{k} \frac{\theta \varepsilon^{k+1}}{2}\|x^{k}\|^2 +  \eta_2^{k} \frac{\theta \varepsilon^{k}}{2} \|\Delta x_{k}\|^2\,,
        \end{aligned}
    \end{equation*}
    where $\eta^{k}$ and $\eta_2^{k}$ are the coefficients that make the last identity hold true for all $k \geq 0$. Applying once again \eqref{eq:condition_for_eps}, it is easy to note that these sequences are indeed infinitesimal, i.e., they converge to zero as $k \to +\infty$. Plugging \eqref{eq:remainder_phi} and \eqref{eq:parallel} into \eqref{eq:dpsi_1} we observe that the term $\frac{\theta \varepsilon^{k+1}}{2}\|x^{k+1}\|^2$ cancels and thus we get for all $k \geq 0$
    \begin{equation}
        \Delta \varphi_{k} + \theta \varepsilon^{k+1} \varphi^{k+1} \leq    \frac{\theta \varepsilon^{k+1}}{2}\left(\|x^*\|^2 - (1 - \eta^{k})\|x^{k}\|^2\right) - \frac{\theta \varepsilon^{k}}{2}\left(1 - \eta_1^{k} - \eta_2^{k}\right) \|\Delta x_{k}\|^2\,.
    \end{equation}
    Taking $k_1'\geq 0$ such that $\eta^{k} \leq 1$ and $\eta_1^{k} + \eta_2^{k} \leq 1$ for all $k \geq k_1'$, we deduce:
    \begin{equation}\label{eq:descent_phi_semifinal}
        \Delta \varphi_{k} + \theta \varepsilon^{k+1} \varphi^{k+1} \leq    \frac{\theta \varepsilon^{k+1}}{2}\left(\|x^*\|^2 - (1 - \eta^{k})\|x^{k}\|^2\right) \, \quad \text{for all} \ k \geq k_1'\,.
    \end{equation}
    Observe now that $\Delta \varphi_{k} + \theta \varepsilon^{k+1} \varphi^{k+1} = (1 + \theta \varepsilon^{k+1})\Delta \varphi_{k} + \theta \varepsilon^{k+1} \varphi^{k}$. Hence, plugging this into the left hand-side of \eqref{eq:descent_phi_semifinal}, dividing by $1 + \theta \varepsilon^{k+1} > 0$ and using that $\frac{\theta \varepsilon^{k+1}}{2} \leq \frac{\theta \varepsilon^{k+1}}{1+\theta \varepsilon^{k+1}} \leq \theta \varepsilon^{k+1}$ for all $k \geq k_1''$ and some $k_1'' \geq 0$ large enough, we obtain the claim with $k_1 :=\max\{k_1', k_1''\}$.
\end{proof}

\subsection{Main Convergence Result}
As in the continuous time approach, we combine the two preliminary lemmas on $(\varphi^k)_k$ and $(\psi^k)_k$ to get a converging result with a general rate expressed in terms of $(\varepsilon^k)_k$. We do so by utilizing the discrete Gronwall inequality. In Section \ref{sec:choice_of_eps} below, we turn this inequality into asymptotic rates for several choices of the parameters relying on Chung's Lemma, see Appendix \ref{sec:chung_lemma}.

\begin{theorem}[Convergence statement]\label{thm:convergence_feg}
	Let $(x^k)_k$, $(y^{k+1})_k$ be generated by Algorithm \ref{alg:gfeg} with $\theta \in \big(0, \frac{1}{L}\big)$ and $(\varepsilon^k)_k$ satisfying \eqref{eq:condition_for_eps}. Then $(x^k)_k$ converges strongly to the minimum norm solution to \eqref{eq:intro_monotone_inclusion} as $k \rightarrow +\infty$, and the convergence rate of the residual results from the following statement: For any $\delta \in (0,1)$, there exist $\widetilde C_\delta \geq 0$ and $k_0 \geq 0$ such that
	\begin{equation}\label{eq:rate_feg_abstract}
		\psi^{k+1} \leq  \frac{\psi^{k_0}}{\gamma^k} + \frac{\widetilde C_\delta}{\gamma^k}\sum_{j = k_0}^k \gamma^j \varepsilon^j \left(\frac{\varepsilon^{j+1}-\varepsilon^j}{\theta \varepsilon^j}\right)^2\ \quad \text{for all} \ k \geq k_0\,,
	\end{equation}
      where $(\gamma^k)_k$ is the sequence defined by $\gamma^k:=\prod_{l=k_0}^k(1-2(1-\delta)\theta \varepsilon^l)^{-1}$.
\end{theorem}
\begin{proof}
	The result is based on Lemma \ref{lem:strong_convergence_explicit} and Lemma \ref{lem:rate_explicit}. From Lemma \ref{lem:strong_convergence_explicit}, we deduce that $(x^k)_k$ is bounded. Thus, Lemma \ref{lem:rate_explicit} applies and yields, together with Lemma \ref{lem:limsup_discrete}, the abstract rate \eqref{eq:rate_feg_abstract}, where $\widetilde C_\delta:= \frac{C_\delta}{2(1-\delta)}$, and $k_0 \geq 1$ is such that, in addition to \eqref{eq:decrease_psi_discrete}, $\varepsilon^k < \frac{1}{2(1-\delta)}$ for all $k \geq k_0$. In particular, since \eqref{eq:condition_for_eps} holds, we get $\psi^k\to 0$ for $k \to +\infty$ by applying the second statement of Lemma \ref{lem:limsup_discrete}. The strong convergence of $(x^k)_k$ to the element with the minimum norm in $\zer M$ follows from the same argument as in Theorem \ref{thm:convergence_cont}, by using Lemma \ref{lem:strong_convergence_explicit} and the second statement of Lemma \ref{lem:limsup_discrete}.
\end{proof}

\begin{remark}[Convergence of $(y^k)_k$]
    Since $(x^k)_k$ converges strongly to the minimum norm solution to \eqref{eq:intro_monotone_inclusion} as $k \rightarrow +\infty$ and $M(x^k)\to 0$ as $k\to +\infty$, $(y^{k+1})_k$ converges also strongly to the same point. 
\end{remark}

\subsection{Choices of \texorpdfstring{$\varepsilon^k$}{epsilon}}\label{sec:choice_of_eps}

An immediate consequence of our approach is a straightforward analysis of Algorithm \ref{alg:gfeg} with several choices of the parameter $(\varepsilon^k)_k$. We only have to turn \eqref{eq:rate_feg_abstract} into explicit rates in each case. We start with the most natural choice.

\paragraph{Linear:} In this section, we make the following choice:
\begin{equation}\label{eq:eps_linear}
	\varepsilon^k:=\frac{\alpha}{\theta(k + \beta)}\, \quad \text{for all} \ k \geq 0, \ \mbox{and} \ \alpha, \beta >0\,,
\end{equation}
which can be understood as a discrete time counterpart of $\varepsilon(t):=\frac{\alpha}{t}$. Specifically, since $\theta$ represents the temporal step size, the term $t^k := \theta(k + \beta)$ represents the discrete time, with $\theta \beta = t^0$ as the starting time $t_0$ in \eqref{eq:tychonov_flow}. The convergence result aligns with its continuous time counterpart stated in Corollary \ref{cor:continuous_optimal_rate}.

\begin{corollary}\label{cor:rate_linear}
	Let $(x^k)_k$, $(y^{k+1})_k$ be generated by Algorithm \ref{alg:gfeg} with $\theta \in \big(0, \frac{1}{L}\big)$ and $(\varepsilon^k)_k$ defined as in \eqref{eq:eps_linear} with $\alpha, \beta >0$. Then $(x^k)_k$ converges strongly to the minimum norm solution to \eqref{eq:intro_monotone_inclusion} as $k \rightarrow +\infty$, the residual showing the following convergence rate:
	\begin{enumerate}
            \item If $\alpha > 1$, then $\|M(x^k)\|  = O(k^{-1})$ as $k\to +\infty$;
		\item If $\alpha \leq 1$, then $\| M(x^k)\| = O(k^{-\alpha (1-\delta)})$ for any $\delta \in (0, 1)$, as $k\to +\infty$.
	\end{enumerate}
\end{corollary}
\begin{proof}
  Let $\delta \in (0, 1)$ and $0 < \hat \delta < \delta$. Lemma \ref{lem:rate_explicit}  yields $C_{\hat \delta} \geq 0$  and $k_0 \geq 0$, such that
    \begin{equation}
        	\Delta \psi_k + 2(1- \hat \delta) \varepsilon^k\theta  \psi^k \leq 2(1-\hat \delta)\varepsilon^k \theta h^k\, \quad \mbox{for all} \ k \geq k_0,
    \end{equation}
    where $h^k:= \frac{C_{\hat \delta}}{2(1- \hat \delta)}\left(\frac{\Delta \varepsilon_k}{\theta \varepsilon^k}\right)^2$. Using the definition of $\varepsilon^k$ and of $h^k$, we get $h^k=\frac{C_{\hat \delta}}{2(1- \hat \delta) \theta^2(k + \beta + 1)^2}$, and thus
	\begin{equation}
		\begin{aligned}
			\psi^{k+1} & \leq \left(1- \frac{2 \alpha  (1-\hat \delta) }{k + \beta }\right)\psi^k + \frac{\alpha C_{\hat \delta}}{(k + \beta )\theta^2 (k + \beta + 1)^2}\\
			& =  \bigg(1- \frac{2\alpha  (1- \hat \delta)}{k}\underbrace{\frac{k}{k + \beta }}_{=:\xi^k}\bigg) \psi^k + \frac{\alpha C_{\hat \delta}}{\theta^2} \frac{1}{k^3} \frac{k^3}{(k + \beta)(k + \beta + 1)^2}\, \quad \mbox{for all} \ k \geq k_0\,.
		\end{aligned}
	\end{equation}
   Since $\xi^k \to 1$ as $k  \to +\infty$, there exists $k_1 \geq k_0$ such that
    \begin{equation}
        \psi^{k+1}\leq \left(1- \frac{2\alpha(1-\delta)}{k}\right)\psi^{k} + \frac{\alpha C_{\hat \delta}}{\theta^2}  \frac{1}{k^3}\,  \quad \text{for all} \ k \geq k_1.
    \end{equation}
    We are in the setting of Lemma \ref{lem:chung_1}. If $\alpha  >1$, choosing $\delta\in (0, 1)$ such that $2\alpha (1-\delta) > 2$, we get $\psi^k = O(k^{-2})$ as $k\to +\infty$. If $\alpha \leq 1$, then $2\alpha (1-\delta) < 2$ and thus $\psi^k = O(k^{-2\alpha (1-\delta)})$. To turn these into rates for $\|M(x^k)\|$, observe that $(2\psi^k)^\frac{1}{2} + O(k^{-1})\geq \|M(x^k)\|$ as $k \rightarrow +\infty$, since $x^k$ is bounded and $\varepsilon^k = O(k^{-1})$ as $k \to +\infty$.
\end{proof}

\paragraph{Power:} Our second choice reinforces the influence of the anchor term as follows
\begin{equation}\label{eq:choice_eps_power}
	\varepsilon^k:=\frac{\alpha}{(k + \beta)^\eta} \quad \text{for all} \ k \geq 0, \ \mbox{and} \ \alpha, \beta >0\,, \ \eta \in (0, 1)\,.
\end{equation}
Simple calculation shows that in this case the conditions in \eqref{eq:condition_for_eps} are also satisfied. However, the convergence rate we are able to obtain gets worse.

\begin{corollary}\label{cor:rate_power}
	Let $(x^k)_k$, $(y^{k+1})_k$ be generated by Algorithm \ref{alg:gfeg} with $\theta \in \big(0, \frac{1}{L}\big)$ and $(\varepsilon^k)_k$ defined as in \eqref{eq:choice_eps_power}. Then $(x^k)_k$ converges strongly to the minimum norm solution to \eqref{eq:intro_monotone_inclusion} as $k \to +\infty$, the residual showing the following convergence rate
 \begin{equation}
     \|M(x^k)\| = O(k^{-\eta})\, \quad \text{as} \ k \to +\infty\,.
 \end{equation}
 \end{corollary}
\begin{proof}
	The proof follows the same argument of Corollary \ref{cor:rate_linear} and therefore we only state the differences. In this case, we have once again $h^k = O(k^{-2})$ for $k\to +\infty$ and thus \eqref{eq:decrease_psi_discrete} writes as
	\begin{equation}
		\psi^{k+1} \leq \left(1 - \frac{c}{k^\eta} \right)\psi^k + \frac{C}{k^{2 + \eta}} \quad \text{for all} \ k \geq k_1\,,
	\end{equation}
	for some $c, C>0$ and $k_1 \geq 0$ large enough. We now make use of Lemma \ref{lem:chung_2} to obtain $\psi^k = O(k^{-2})$ as $k\to  +\infty$. Since $(2\psi^k)^\frac{1}{2} + O(k^{-\eta})\geq \|M(x^k)\|$, we get $\|M(x^k)\|= O(k^{-\eta})$ as $k\to +\infty$.
\end{proof}

\begin{remark}
    The asymptotic behaviour of the abstract rate in \eqref{eq:rate_feg_abstract} is the same as the one inferred for $(\psi^k)_k$. Indeed, if we denote by $r^{k+1}$ the right hand-side of \eqref{eq:rate_feg_abstract}, which satisfies $\psi^{k+1} \leq r^{k+1}$ for all $k \geq k_0$, it is easy to note that the sequence $(r^k)_k$ satisfies the following recursion
 \begin{equation}\label{eq:recursion_worst_case_rate}
		r^{k+1} =  (1- 2(1-\delta)\varepsilon^{k}\theta ) r^{k} +  C_\delta \theta \varepsilon^{k} \left(\frac{\varepsilon^{k+1}-\varepsilon^k}{ \theta \varepsilon^k}\right)^2 \, \quad \text{for all} \ k \geq k_0\,.
	\end{equation}
 Therefore, applying Chung's Lemma on this expression, we derive the same results as in Corollary \ref{cor:rate_linear} and Corollary \ref{cor:rate_power}.
 \end{remark}

\paragraph{Miscellaneous:} Our approach to deriving convergence rates leverages inequality \eqref{eq:decrease_psi_discrete} in combination with Chung's Lemma. Originally introduced by Chung in \cite{chung54} to study the stochastic gradient descent method, this lemma has since become a key tool in the stochastic optimization literature. It has been extensively used, and efforts have been made to derive nonlinear and non-asymptotic extensions \cite{mb11, ays20, jlmq24}. We can leverage these results to obtain several analytical bounds on $(\psi^k)_k$. However, due to space constraints, we limit our focus here to present the explicit expansion \eqref{eq:rate_feg_abstract} of the worst-case rate numerically. We consider four different parameter sequences $(\varepsilon^k)_k$, and plot the corresponding sequence $(r^k)_k$ in Figure \ref{fig:rates}.

We observe in particular that the choice $\varepsilon^k = \alpha \frac{2}{\pi} \frac{\arctan(k)}{\theta(k + \beta)}$ behaves like $\varepsilon^k = \frac{\alpha}{\theta(k + \beta)}$ for $\alpha>0$ and $\beta >0$ and in particular raises $\|M(x^k)\| = O(k^{-2})$ as $k \to +\infty$ for $\alpha > 1$. This distinction between $\alpha>1$ and $\alpha < 1$ does not seem to be present for $\varepsilon^k = \frac{\alpha}{k} \log (k)$, which always yields $r^k =O(k^{-2})$ as $k\to +\infty$. However, this choice corrupts the residual giving $\|M(x^k)\| = O (  \log(k)^{\frac{1}{2}} k^{-1} )$ as $k \to +\infty$. The same applies for $\varepsilon^k = \frac{1}{\log(k)}$, which despite seemingly yielding $r^k = O (k^{-2} )$ as $k \to +\infty$, would only give $\|M(x^k)\|=O ( \log(k)^{-\frac{1}{2}} )$ as $k\to +\infty$. In conclusion it appears that the best parameter choices in terms of worst-case residuals are of the form $\varepsilon^k := \frac{\alpha^k}{\theta(k + \beta)}$ with $\alpha^k = O(1)$ and $\alpha^k \geq 1$ for $k$ large enough.

\begin{figure}
    \centering
    \includegraphics[width=0.8\linewidth]{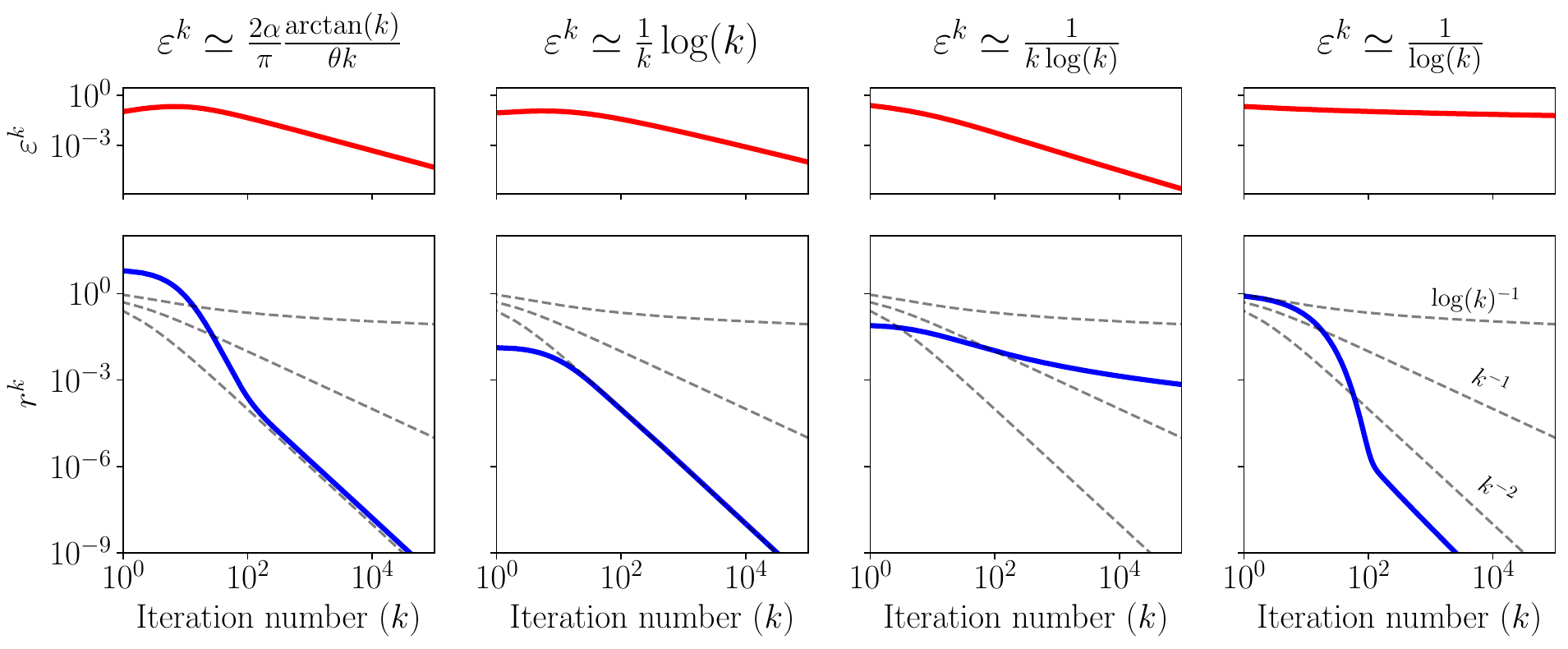}
    \caption{Numerical study of the worst-case rate. First row: the parameter choice $(\varepsilon^k)_k$. Second row: The sequence $(r^k)_k$ in \eqref{eq:recursion_worst_case_rate} as a function of $k$, the upper bound for $\psi^k = \frac{1}{2}\|M(x^k) + \varepsilon^k x^k \|^2$.} 
    \label{fig:rates}
\end{figure}

\section{Numerical Experiments}\label{sec:numerics}

In this section, we present our numerical experiments, which are performed in Python on a 12th-Gen.~Intel(R) Core(TM) i7--1255U, $1.70$--$4.70$ GHz laptop with $16$ Gb of RAM and are available for reproducibility at \href{https://github.com/echnen/general-anchored-extragradient}{https://github.com/echnen/general-anchored-extragradient}.

\paragraph{Preparation:} For each $L$-Lipschitz continuous operator $M\colon H\to H$ with $\zer M\neq \emptyset$, we initialize $x^0 \in H$ and consider the methods listed in Table \ref{tab:eag_type_methods} with the following specifications: 
\begin{itemize}
	\item \textit{(Extra-Anchored-Gradient Method with Constant step-sizes (EAG-C))} Defined as in Table \ref{tab:eag_type_methods} and with step-size $\eta = \frac{1}{8L}$. 
	\item \textit{(Extra-Anchored-Gradient Method with Variable step-sizes (EAG-V))} Defined as in Table \ref{tab:eag_type_methods} and variable step-sizes $\eta^k = \theta^k$ defined by the following recursion:
 \begin{equation}\label{eq:step-size_eag_v}
     \theta^{k+1} = \theta^k\left( 1 - \frac{(\theta^k)^2 L^2}{(k+1)(k+2)(1-(\theta^k)^2 L^2)}\right)\, \quad \text{for} \ k \geq 0\,.
 \end{equation}
    \item \textit{(Fast Extra-Gradient Method (FEG))} Defined as in Table \ref{tab:eag_type_methods}.
	\item \textit{(Anchored Popov's Method (APV))} Defined as in Table \ref{tab:eag_type_methods} with $\eta^k = \theta^k$ defined by:
	\begin{equation}\label{eq:step-size_apv}
		 \theta^{k+1} := \frac{(1-(\varepsilon^k)^2 - K (\theta^k)^2)\varepsilon^{k+1}\theta^k}{(1-K(\theta^k)^2)(1-\varepsilon^k)\varepsilon^k}  \quad \text{for} \ k \geq 0, \quad \eta \in \left(0, \frac{1}{2L\sqrt{3}}\right)\,, \quad K :=4L^2\,.
	\end{equation}
\end{itemize}

\subsection{Comparison}\label{sec:num_comparison}

In this section, we compare the performance of Algorithm \ref{alg:gfeg} against the variants listed in Table \ref{tab:eag_type_methods}. To do so, we build a nonlinear operator $M\colon H\to H$ with $H = \R^{2d}$ and $d = 5$ by considering
\begin{equation}
    N := \begin{bmatrix}
        0 & I\\
        -I & 0
    \end{bmatrix}\,, \quad \text{and} \quad x^* \in \R^{2d}\,, \quad b := N x^*\,.
\end{equation}
where $I\in \R^{d\times d}$ is the identity matrix, $0\in \R^{d\times d}$ is the zero matrix and $x^*$ is randomly sampled. This allows us to define the monotone linear operator $x\mapsto Nx - b$ with zero $x^*$. Note that $x^*$ is not necessarily zero, which might induce undesired effects. To induce some nonlinearity, we consider $P\colon H\to H$ to be projection onto a ball with radius $1$ and center $x^*$. Then, we define
\begin{equation}
    M(x):= Nx - b + P(x)\, \quad \text{for all} \ x \in \R^{2d}\,,
\end{equation}
which is indeed monotone and Lipschitz continuous with $M(x^*)=0$. We estimate the Lipschitz constant of $M$ by $L:= \|M\|_2 + 1$.

We compare the methods listed above with Algorithm \ref{alg:gfeg} with $\varepsilon^k=\frac{\alpha}{\theta(k + \beta)}$ and $10$ uniformly spaced choices of $\alpha$ in $(10^{-3}, 1]$ (including the edge case $\alpha = 1$) and $10$ uniformly spaced choices of $\alpha$ in $(1, 2]$. We further consider $4$ different choices of $\beta$: $\beta = 1, 2, 5, 50$. We put particular emphasis on the choice $\alpha = 1$, which according to Corollary \ref{cor:rate_linear}, separates the sub-optimal regime with rate $O(k^{- \alpha  (1-\delta)})$ for any $\delta \in (0, 1)$, with the optimal one of order $O(k^{-1})$. We plot the results in Figure \ref{fig:experiment_1} putting particular care in emphasizing the choice $\alpha = 1$, which lies on the middle of the range of considered parameters.   

\begin{figure}[t]
	\includegraphics[width=0.6\linewidth]{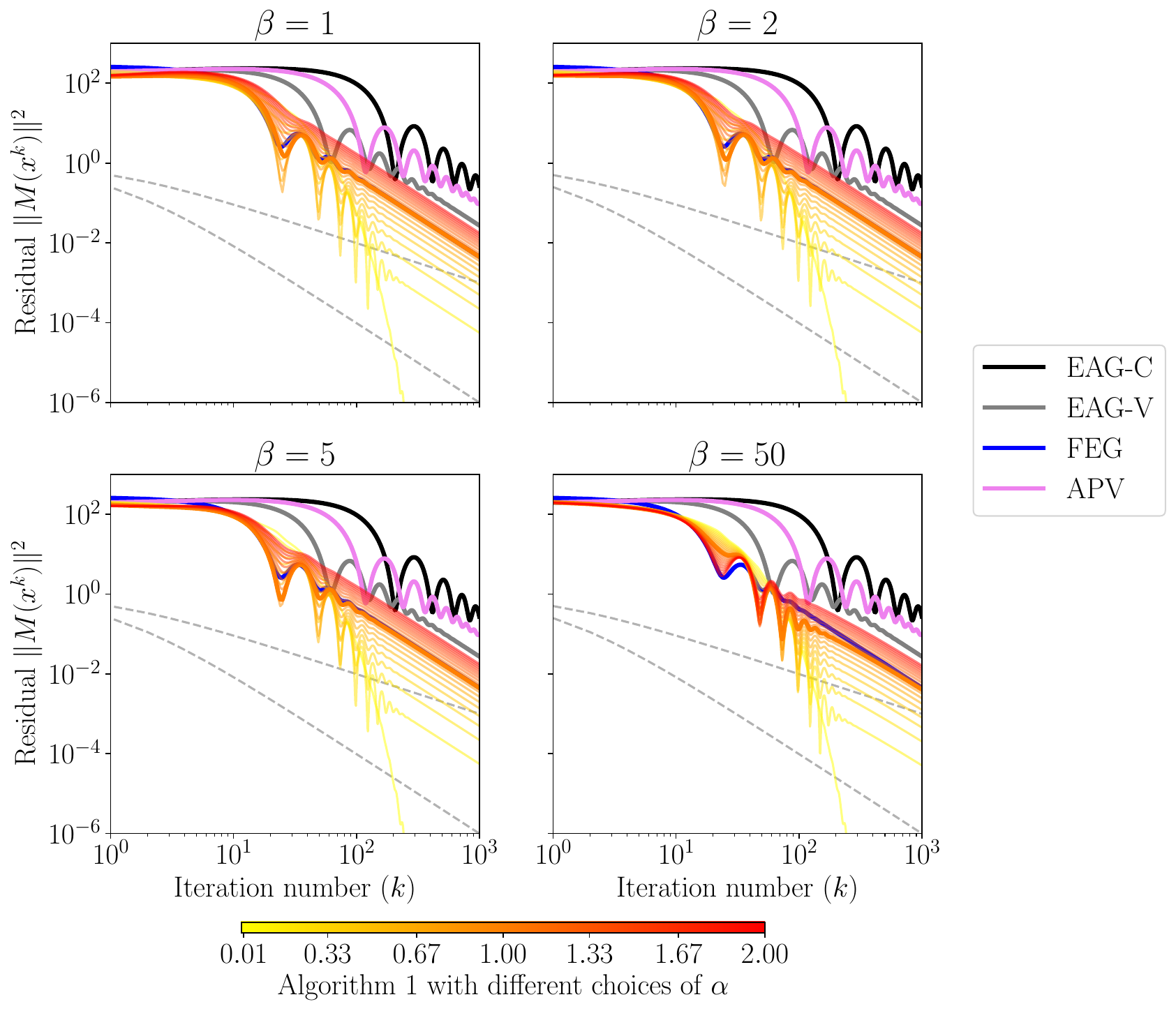}
	\caption{Results of the experiment in Section \ref{sec:num_comparison}. Comparing Algorithm \ref{alg:gfeg} with EAG-C, EAG-V, FEG and APV.}
	\label{fig:experiment_1}
\end{figure}

\paragraph{Comments:} Figure \ref{fig:experiment_1} reveals that Algorithm \ref{alg:gfeg} is particularly competitive against baseline choices. Notably, when $\alpha = 1$, the algorithm slightly outperforms FEG. Even more interesting is that for values of $\alpha < 1$, the algorithm shows even greater improvements, with a significant performance boost when $\alpha \simeq 10^{-3}$, despite leading to suboptimal worst-case asymptotic rates. This behavior could be attributed to the well-known decelerating effect of anchoring accelerations. While these accelerations achieve optimal rates, they often result in slower performance in practical cases. Thus, selecting $\alpha \simeq 0$ can lead to a noticeable advantage.

\subsection{A Practical Enhancement}\label{sec:num_heuristics}

Motivated by the observations in Section \ref{sec:num_comparison}, we propose here two parameter choices that prove particularly advantageous in practice. These are designed to combine the desirable fast convergence behaviour of the vanilla EG method without anchoring, and the improved worst-case rate with guaranteed (strong) convergence to specific points in $\zer M$ of anchoring mechanisms. We do so by considering:
\begin{equation}
   \varepsilon^k = \frac{\alpha^k}{\theta(k + \beta)}\, \quad \text{for all} \ k \geq 0\,, 
\end{equation}
where $\alpha^k$ is set to be either $\alpha^k:=\frac{2}{\pi} \arctan(M k)$, with $M\in \R_+$ (we set $M=10^{-3}$), or $\alpha^k$ to be equal to the same as before but further corrupted with Gaussian noise of variance $\sigma^k = \frac{1}{k}$ for all $k \geq 1$. We compare these two methods, denoted by G-EAG-C1 and G-EAG-C2, with the methods listed in Table \ref{tab:eag_type_methods} with the specifications presented at the beginning of Section \ref{sec:numerics}. The results are plotted in Figure \ref{fig:experiment_2}.

\begin{figure}[t]
	\includegraphics[width=0.6\linewidth]{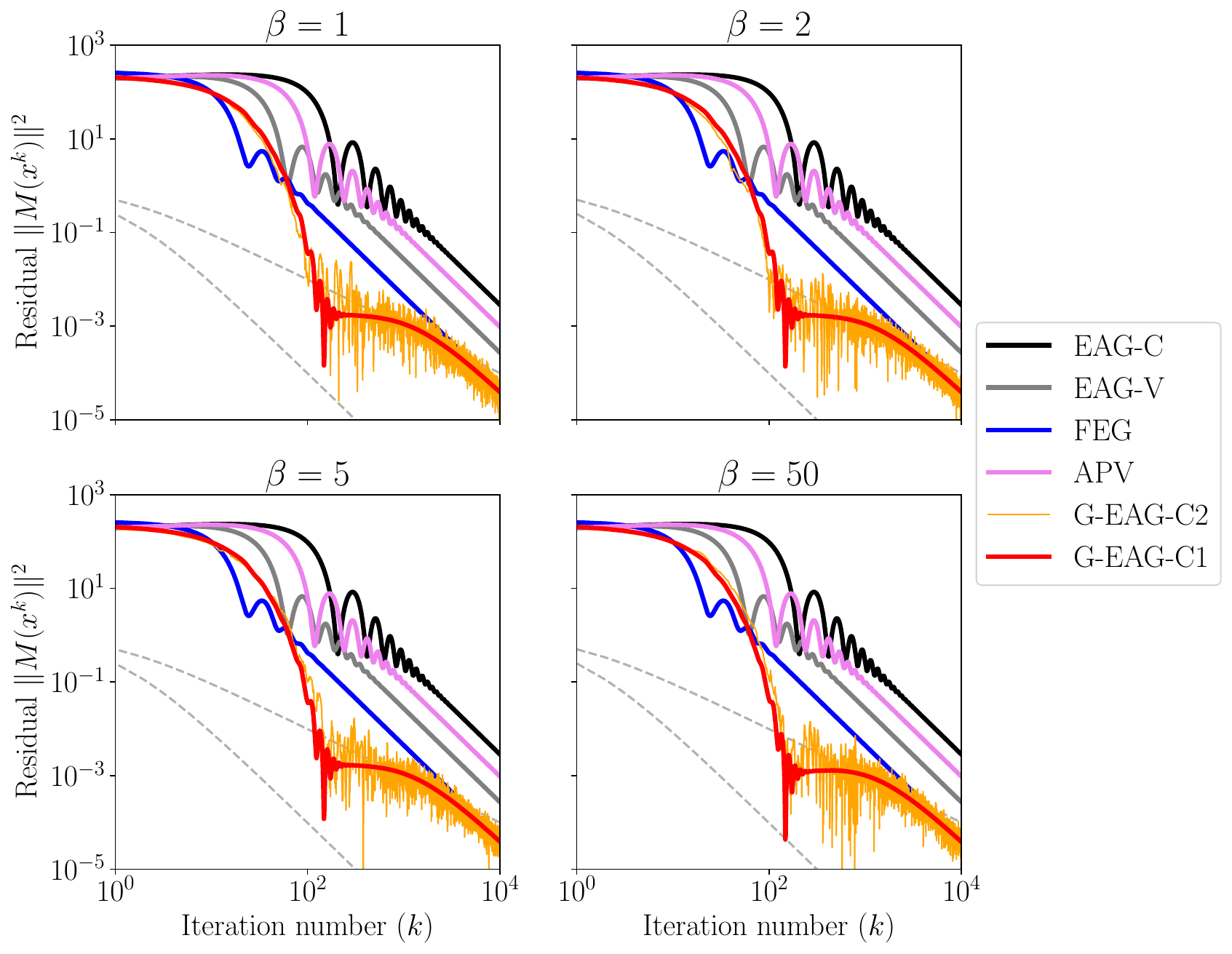}
	\caption{Result of experiment in Section \ref{sec:num_comparison}. Comparing Algorithm \ref{alg:gfeg} with EAG-C, EAG-V, FEG and APV.}
	\label{fig:experiment_2}
\end{figure}

\paragraph{Comments:} From Figure \ref{fig:experiment_2} we can observe that Algorithm \ref{alg:gfeg} with the enhancements proposed in Section \ref{sec:num_heuristics} seemingly outperforms the considered methods reaching the accuracy of $10^{-4}$ in about half the maximum of iterations allowed. The deterministic choice of $\alpha^k$, i.e., $\alpha^k = \frac{2}{\pi} \arctan(Mk)$ is particularly interesting as this choice falls into the framework of Theorem \ref{thm:convergence_feg} and can also be shown to yield the rate $O(k^{-1})$ of the residual, see Section \ref{sec:choice_of_eps} for numerical evidence of this fact.

\subsection{An Infinite Dimensional Example}\label{sec:num_infinite_dimensional}

In this section, we test Algorithm \ref{alg:gfeg} to solve an infinite dimensional problem on $H:= \ell^2$, the space of square summable sequences. Specifically, we consider the following operator:
\begin{equation}
    M(x) = x - Sx - b\,, \quad \text{where} \quad S(x_0, \dots) = (0, x_0, \dots)\,, \quad \text{for all} \ x := (x_0, x_1,\dots) \in H\,, 
\end{equation}
i.e., $S$ is the right-shift operator and $b\in H$ is $b:=(1, -1, 0,\dots)$. Note that since $S$ is nonexpansive, $M$ is indeed monotone and $2$-Lipschitz continuous. Additionally, solving the system $M(x)=0$ we obtain $x^*:=(1, 0, \dots)$.

To address this example, we set a maximum number of iterations $K \in \N$ and initialize $x^0 \in \Span\{\delta^0, \delta^1\}$ randomly. In this way, $x_1, \dots, x_K$ will be always contained in $\Span\{\delta^0, \dots, \delta^{2(K + 1)}\}\simeq \R^{2(K+1)}$. Specifically, we set $K:=10^3$. We test the methods described in Section \ref{sec:num_comparison} with an additional method: The naive implementation of EG, which is known to produce a sequence that only converges \emph{weakly} to the zero of $M$ with the rate $o(k^{-\frac{1}{2}})$. For this example, we plot both the residual as a function of the iterations, and the distance to the solution. These results are shown in Figure \ref{fig:infinite_dimensional}.

\begin{figure}[t]
    \centering
    \begin{subfigure}{0.415\linewidth}
        \includegraphics[width=\linewidth]{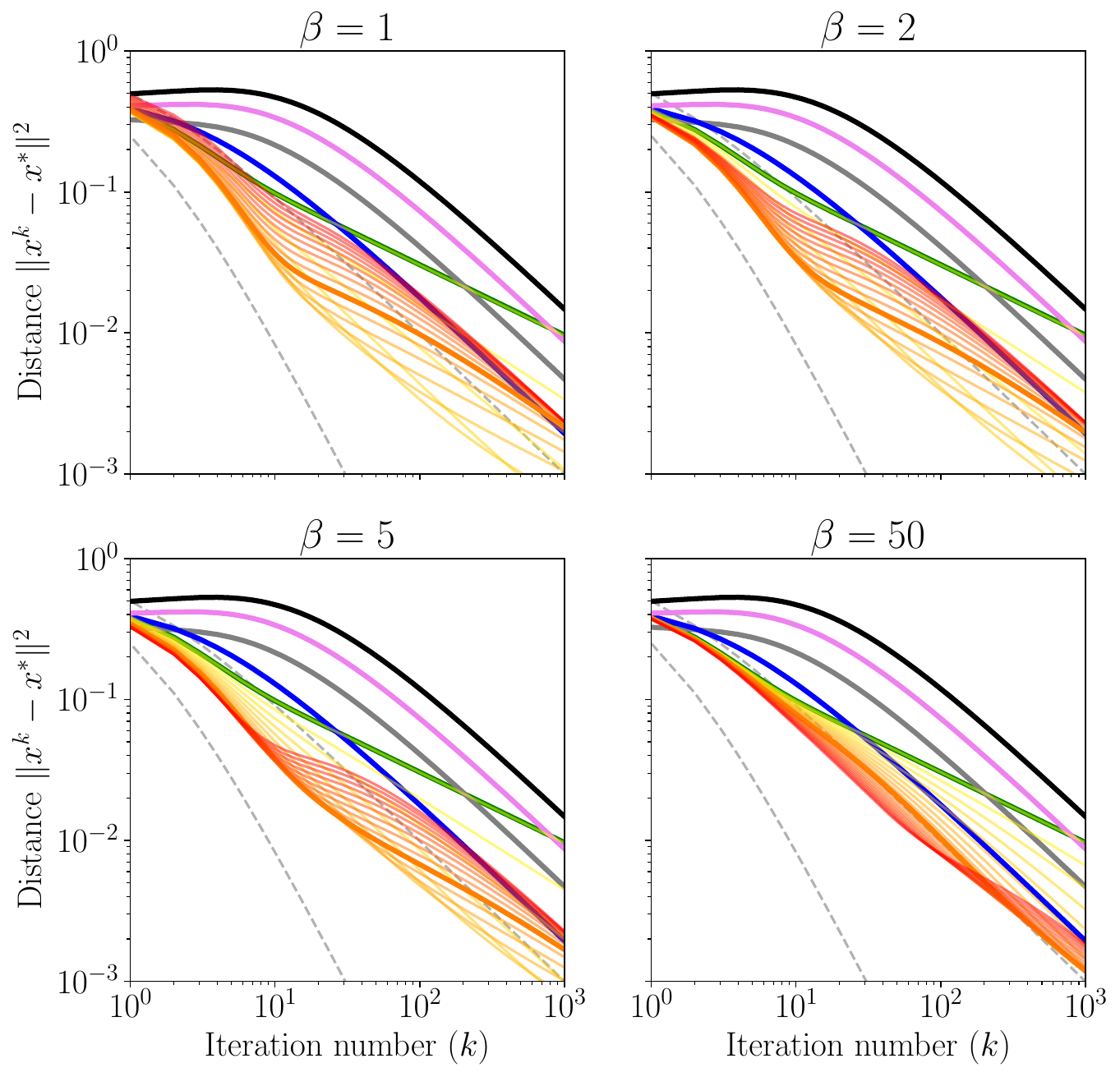}
        \caption{Distance to solution}
        \label{fig:infinite_dimensional_dists}
    \end{subfigure}
    \begin{subfigure}{0.575\linewidth}
        \includegraphics[width=\linewidth]{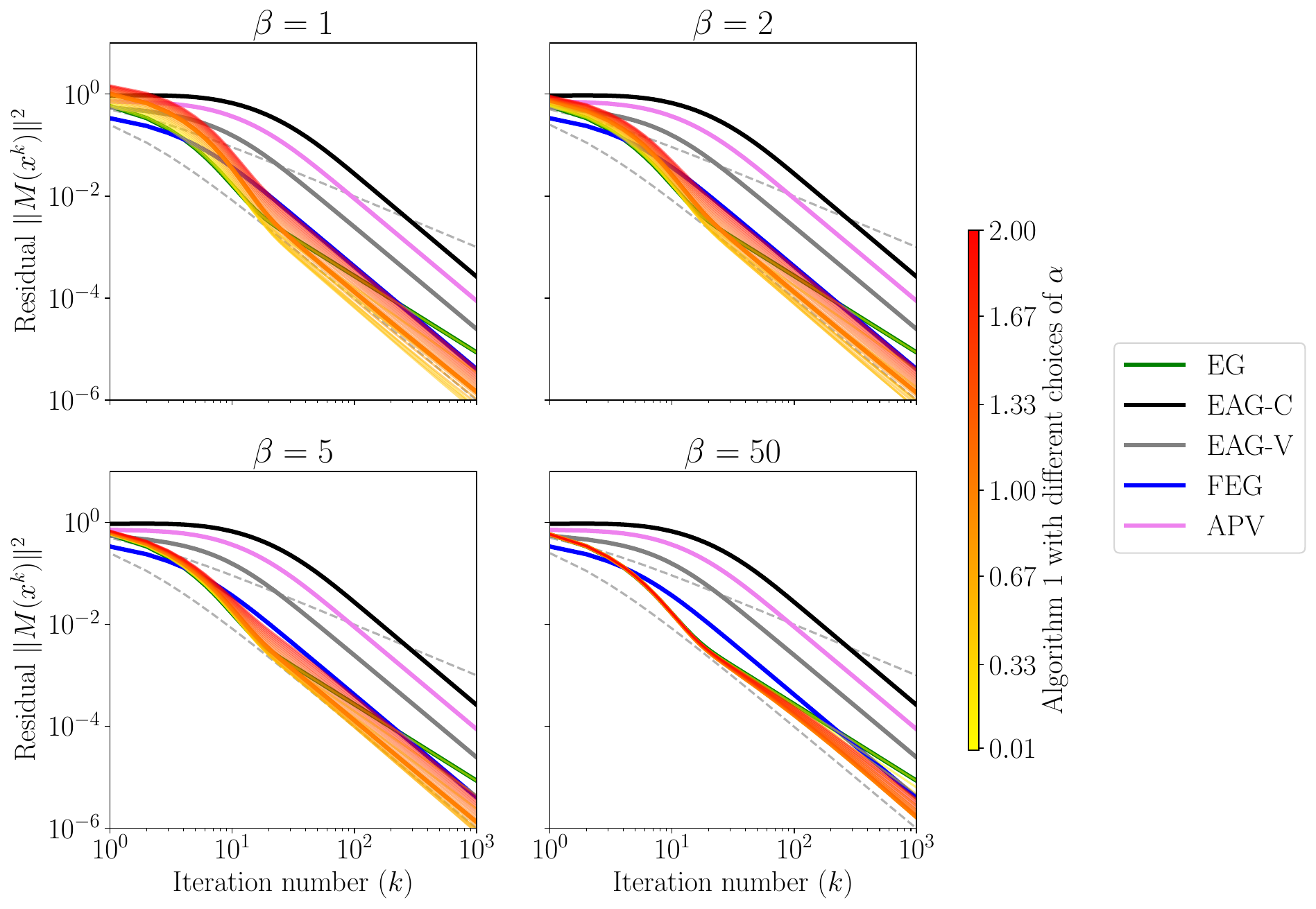}
        \caption{Residual}
    \end{subfigure}
    \caption{Results of the Experiment in Section \ref{sec:num_infinite_dimensional}. Comparison of EAG-type methods to solve an infinite dimensional problem in $\ell^2$.}
    \label{fig:infinite_dimensional}
\end{figure}

\paragraph{Comments:} From Figure \ref{fig:infinite_dimensional}, we see that the proposed method is able to yield excellent results even in an infinitely dimensional setting, with the case $\alpha = 1$ (colored in orange), consistently matching and sometimes surpassing FEG. Conversely, the naive implementation of EG produces a sequence that shows the typical $o(k^{-\frac{1}{2}})$ convergence rate. Algorithm \ref{alg:gfeg} is particularly competitive also in terms of distance to solution, as shown in Figure \ref{fig:infinite_dimensional_dists}. This example also seems to suggest that in this case EG only converges weakly, although this would become clearer letting the algorithm run for a much higher number of iterations.

\section{Conclusion}
In this paper, we proposed a new Extra-Anchored-Gradient-type scheme with a general choice of parameter sequences $(\varepsilon^k)_k$, which only has to satisfy \eqref{eq:condition_for_eps}. Our approach is profoundly based on the continuous time analysis of Tikhonov regularized monotone operator flows, which guide our convergence proof while suggesting critical design choices of the algorithm itself. In particular, it allows us to establish the strong convergence to a specific point in the solution set without relying on other systems dynamics. Although we allow for a general class of parameters, the main drawback of our approach is the lack of an explicit constant in the abstract rate \eqref{eq:rate_feg_abstract}. We suggest a possible approach to address this issue in Remark \ref{rmk:alpha_one}. Indeed, despite our numerical experiments suggest that Algorithm \ref{alg:gfeg} with $\varepsilon^k = \frac{\alpha}{\theta(k+\beta)}$ and $\alpha \simeq 1$ seems to slightly improve over FEG, it remains an open question whether its worst-case rate actually improves upon the one of FEG with this specific choice of $(\varepsilon^k)_k$.

\appendix
\section{Appendix}\label{sec:appendix}
%\addcontentsline{toc}{section}{Appendix}
\renewcommand{\thesubsection}{\Alph{subsection}}

\subsection{Discrete Calculus Rules}\label{sec:discrete_calculus}

Let $a:=(a^k)_k$, $b:=(b^k)_k$, and $c:=(c^k)_k$, be three real sequences. The product of two sequences $a$ and $b$ is denoted by $ab :=(a^k b^k)_k$. As usual, if $a=b$, we denote $a^2 = ab$. For all $k\geq 0$ it holds
\begin{equation}\label{eq:discrete_calculus_rules}
	\begin{aligned}
		&\Delta (ab)_k = \Delta a_k b^k +  a^{k+1}\Delta b_k\,,\\
		&\Delta (abc)_k = \Delta a_k b^k c^k + a^{k+1}\Delta b^k c^k + a^{k+1}b^{k+1}\Delta c_k\,,\\
		& \Delta a^2_k = 2(a^{k+1}-a^k)a^k +(\Delta a^k)^2 = 2(a^{k+1} - a^k)a^{k+1} - (\Delta a_k)^2 \,.
	\end{aligned}
\end{equation}
This formalism can be readily extended to sequences in a Hilbert space $H$, with products replaced by inner products. For instance, if $(z^k)_k$ is a sequence in $H$, and $\zeta^k := \frac{1}{2}\|z^k\|^2$ for all $k \geq 0$, we have for all $k \geq 0$
\begin{equation}
    \Delta \zeta_k =  \bra \Delta z_k, z^k \ket + \frac{1}{2} \| \Delta z_k \|^2 = \bra \Delta z_k, z^{k+1} \ket  - \frac{1}{2} \| \Delta z_k \|^2\,.
\end{equation}
Similar rules as those in \eqref{eq:discrete_calculus_rules} apply for products of sequences in $H$ and in $\R$, for instance if $(\varepsilon^k)_k$ is a real sequence, then for all $k \geq 0$
\begin{equation}
    \Delta (\varepsilon z)_k = \Delta \varepsilon_k z^k +  \varepsilon^{k+1}\Delta z_k = \Delta z_k \varepsilon^k +  z^{k+1}\Delta \varepsilon_k\,.
\end{equation}

\subsection{Gronwall in Continuous and Discrete Time}

The following result is an application of Grownwall's inequality, see for instance \cite[Lemma 1]{cps08}.
\begin{lemma}\label{lem:limsup_continuous}
Let $f \colon \left[t_{0} , + \infty \right) \to \R$ be a locally absolutely continuous function, $h \colon \left[t_{0} , + \infty \right) \to \R$ a bounded function and $\varepsilon \colon \left[t_{0} , + \infty \right) \to (0,+\infty)$ a locally integrable function. Suppose that
\begin{equation}
\dot{f}(t)+ \varepsilon (t) f (t) \leq \varepsilon (t) h (t) \quad \mbox{for a.e.} \ t \geq t_0\,.
\end{equation}
Denote by $\gamma(t):=\exp\Big(\int_{t_0}^t \varepsilon(s)ds \Big)$. Then,
\begin{equation}
    f(t) \leq \frac{f(t_0)}{\gamma(t)} + \frac{1}{\gamma(t)} \int_{t_0}^t \varepsilon(s) \gamma(s) h(s) ds\, \quad \text{for all} \ t \geq t_0\,.
\end{equation}
In particular, if $\int_{t_{0}}^{+ \infty} \varepsilon (t) dt = +\infty$, then 
\begin{equation}
\limsup_{t \to + \infty} f \left( t \right) \leq \limsup_{t \to + \infty} h \left( t \right).    
\end{equation}
\end{lemma}

The following result is the discrete counterpart of Lemma \ref{lem:limsup_continuous}. For this result, we only present the derivation of the explicit expansion \eqref{eq:explicit_expansion}, as \eqref{eq:xu} can be found, e.g., in \cite[Lemma 2.5]{xu02}. 
\begin{lemma}\label{lem:limsup_discrete}
Let $(f^k)_k$ and $(h^k)_k$ be two sequences in $\R$ and $(\varepsilon^k)_k$ be a nonnegative sequence satisfying
\begin{equation}
    f^{k+1} - f^k + \varepsilon^k f^k \leq \varepsilon^k h^{k} \, \quad \text{for all} \ k \geq 0\,.
\end{equation}
Suppose further that $\varepsilon^k < 1$ for all $k \geq 0$. Then it holds 
\begin{equation}\label{eq:explicit_expansion}
    f^{k+1} \leq \frac{f^0}{\gamma^k}  + \frac{1}{\gamma^k} \sum_{j=0}^{k} \gamma^j \varepsilon^j h^j\, \quad \text{for all} \ k \geq 0\,,
\end{equation}
where $\gamma^k := \prod_{\ell = 0}^k (1-\varepsilon^\ell)^{-1}$.  In particular, if $\sum_{k=0}^{+\infty}\varepsilon^k = +\infty$, it holds
\begin{equation}\label{eq:xu}
    \limsup_{k\to +\infty} f^k \leq \limsup_{k\to +\infty} h^k\,.
\end{equation}
\end{lemma}
\begin{proof}[Proof of \eqref{eq:explicit_expansion}]
    Let $k \geq 0$, and observe that
    \begin{equation}
        \begin{aligned}
           f^{k+1} &\leq (1 - \varepsilon^k) f^k + \varepsilon^k h^k\\
           & \leq (1 - \varepsilon^k) (1 - \varepsilon^{k-1}) f^{k-1} + (1 - \varepsilon^k) \varepsilon^{k-1} h^{k-1} + \varepsilon^k h^k\\
           & \dots\\
           & \leq \prod_{\ell = 0}^k (1 - \varepsilon^j) f^0 + \sum_{j =0}^k \prod_{\ell = j + 1}^k (1 - \varepsilon^\ell) \varepsilon^j h^j\,,
        \end{aligned}
    \end{equation}
    where for convenience we set $\prod_{\ell = k + 1}^k (1 - \varepsilon^\ell) = 1$. Using the definition of $\gamma^k$ we get \eqref{eq:explicit_expansion}.
\end{proof}
The term $\gamma^k$ can indeed be understood as the discrete counterpart of $t \mapsto \gamma(t)$, since
\begin{equation}
    \gamma^k = \exp\Bigg(\log\left( \prod_{\ell = 0}^k (1-\varepsilon^j)^{-1}\right)\Bigg) =  \exp\bigg(-\sum_{\ell = 0}^k \log\left(1-\varepsilon^j\right) \bigg) \simeq  \exp \bigg(\sum_{\ell = 0}^k \varepsilon^\ell\bigg)\,,
\end{equation}
by Taylor expansion of $t \mapsto \log(1 - t)$ around $0$.

\subsection{Chung's Lemma}\label{sec:chung_lemma}

For completeness, we include below two variants of Chung's lemma \cite{chung54} as presented in Polyak's monograph \cite[Section 2.2]{polyak87}.

\begin{lemma}[Lemma 4 in \cite{polyak87}]\label{lem:chung_1} Let $(\psi^k)_k$ be a non-negative sequence and let $p, c, C > 0$. Suppose that
\begin{equation}
    \psi^{k+1} \leq \left(1 - \frac{c}{k} \right) \psi^k + \frac{C}{k^{p + 1}}\, \quad \text{for all} \ k \geq 1\,.
\end{equation}
Then, the following holds true as $k \rightarrow +\infty$
\begin{equation}
    \begin{aligned}
        &\psi^k \leq C (c - p)^{-1} k^{-p} + o(k^{-p}), \quad &\text{if} \ c > p\,;\\
        &\psi^k = O(k^{-c} \log (k) ), \quad &\text{if} \ c = p\,;\\
        &\psi^k = O(k^{-c}), \quad &\text{if} \ c < p\,.
    \end{aligned}
\end{equation}
\end{lemma}

\begin{lemma}[Lemma 5 in \cite{polyak87}]\label{lem:chung_2} Let $(\psi^k)_k$ be a non-negative sequence and let $c, C > 0$, and $0 < s < 1$, $s < t$. Suppose that
\begin{equation}
    \psi^{k+1} \leq \left(1 - \frac{c}{k^s} \right) \psi^k + \frac{C}{k^{t}}\, \quad \text{for all} \ k \geq 1\,.
\end{equation}
Then, as $k \rightarrow +\infty$,
\begin{equation}
   \psi^{k} \leq \frac{C}{c} \frac{1}{k^{t - s}} + o\left(\frac{1}{k^{t-s}}\right)\,.
\end{equation}
\end{lemma}

\bibliographystyle{plain} % We choose the "plain" reference style
\bibliography{references} % Entries are in the refs.bib file

\end{document}